\font \sevenrm=cmr7
\font \fiverm=cmr5
\font\sevenbf=cmbx7

\documentclass[11pt, english]{amsart}
\input xy
\xyoption{all}
\usepackage{amsfonts}
\usepackage{amssymb}
\usepackage{amsmath}
\usepackage{color}
\usepackage{graphicx}
\usepackage{xspace}
\usepackage{axodraw}
\usepackage{wasysym}
\usepackage{bm}
\def\oprec{\!\!\joinrel{\ocircle\hskip -12.5pt \prec}\,}
\def\soprec{\,\joinrel{\ocircle\hskip -7.5pt \prec}\,}

\def\diagramme #1{\vskip 4mm \centerline {#1} \vskip 4mm}

 \newcommand{\nc}{\newcommand}
  
\newcommand{\nocomma}{}
\newcommand{\noplus}{}
\newcommand{\nosymbol}{}
\newcommand{\tmdummy}{$\mbox{}$}
\newcommand{\tmem}[1]{{\em #1\/}}
\newcommand{\tmmathbf}[1]{\ensuremath{\boldsymbol{#1}}}
\newcommand{\tmop}[1]{\ensuremath{\operatorname{#1}}}

\newcommand{\setcomp}{\mathbb{SC}}
\newcommand{\T}{\mathcal{T}}
\newcommand{\topo}{\mathbb{T}}
\newcommand{\lin}{\mathcal{L}}
\newcommand{\h}{\mathcal{H}}
\newcommand{\QSym}{\mathbf{QSym}}
\newcommand{\WQSym}{\mathbf{WQSym}}

 {\everymath{\displaystyle\everymath{}}\array}%
 {\endarray}

\newcommand{\tdun}[1]{\begin{picture}(12,5)(-2,-1)
\put(0,0){\circle*{2}}
\put(3,-2){\tiny #1}
\end{picture}}

\newcommand{\tddeux}[2]{\begin{picture}(12,10)(0,-1)
\put(3,0){\circle*{2}}
\put(3,5){\circle*{2}}
\put(3,0){\line(0,1){5}}
\put(6,-2){\tiny #1}
\put(6,4){\tiny #2}
\end{picture}}

\newcommand{\tdtroisun}[3]{\begin{picture}(22,12)(-6,-1)
\put(3,0){\circle*{2}}
\put(6,7){\circle*{2}}
\put(0,7){\circle*{2}}
\put(-0.65,0){$\vee$}
\put(5,-2){\tiny #1}
\put(9,5){\tiny #2}
\put(-7,5){\tiny #3}
\end{picture}}

\newcommand{\tdtroisdeux}[3]{\begin{picture}(12,15)(-2,-1)
\put(0,0){\circle*{2}}
\put(0,5){\circle*{2}}
\put(0,10){\circle*{2}}
\put(0,0){\line(0,1){5}}
\put(0,5){\line(0,1){5}}
\put(3,-2){\tiny #1}
\put(3,3){\tiny #2}
\put(3,9){\tiny #3}
\end{picture}}

\newcommand{\pdtroisun}[3]{\begin{picture}(23,12)(-7,-1)
\put(3,7){\circle*{2}}
\put(-0.65,0){$\wedge$}
\put(6,0){\circle*{2}}
\put(0,0){\circle*{2}}
\put(5,5){\tiny #1}
\put(-7,-2){\tiny #2}
\put(9,-2){\tiny #3}
\end{picture}}

\newtheorem{thm}{Theorem}

\newtheorem{cor}[thm]{Corollary}

\newtheorem{lem}[thm]{Lemma}
\newtheorem{prop}[thm]{Proposition}
\newtheorem{defn}{Definition}
\newtheorem{rmk}[thm]{Remark}

\nc{\comment}[1]{[[{\tt #1}]] }
\nc{\Cal}[1]{{\mathcal {#1}}}
\nc{\mop}[1]{\mathop{\hbox {\rm #1} }\nolimits}
\nc{\gmop}[1]{\mathop{\hbox {\bf #1} }\nolimits}

\nc{\smop}[1]{\mathop{\hbox {\sevenrm #1} }\nolimits}
\nc{\ssmop}[1]{\mathop{\hbox {\fiverm #1} }\nolimits}
\nc{\mopl}[1]{\mathop{\hbox {\rm #1} }\limits}
\nc{\smopl}[1]{\mathop{\hbox {\sevenrm #1} }\limits}
\nc{\ssmopl}[1]{\mathop{\hbox {\fiverm #1} }\limits}
\nc{\frakg}{{\frak g}}
\nc{\g}[1]{{\frak {#1}}}
\def \restr#1{\mathstrut_{\textstyle |}\raise-6pt\hbox{$\scriptstyle #1$}}
\def \srestr#1{\mathstrut_{\scriptstyle |}\hbox to
  -1.5pt{}\raise-4pt\hbox{$\scriptscriptstyle #1$}}
\nc{\wt}{\widetilde} \nc{\wh}{\widehat}

\nc{\redtext}[1]{\textcolor{red}{#1}}
\nc{\bluetext}[1]{\textcolor{blue}{#1}}

\nc\fleche[1]{\mathop{\hbox to #1 mm{\rightarrowfill}}\limits}
\nc{\ignore}[1]{}
\def\semi{\mathrel{\times}\kern -.85pt\joinrel\mathrel{\raise
    1.4pt\hbox{${\scriptscriptstyle |}$}}}
\nc\R{{\mathbb R}}
\nc\N{{\mathbb N}}
\nc\inver{^{-1}}
\nc\point{\hbox{\bf .}}
\nc\un{\hbox{\bf 1}}
\def\link#1#2{\raise -2pt\hbox{$\scriptstyle #1-\!\!-\!\!- #2$}}
\def\slink#1#2{\raise -1.5pt\hbox{$\scriptscriptstyle #1-\!\!\!-\!\!\!- #2$}}

\def\arbreccbig{\,{\scalebox{0.8}{
 \begin{picture}(48,98) (349,-205)
    \SetWidth{2}
    \SetColor{Black}
    \Vertex(375,-202){12}
    \SetColor{White}
    \Vertex(371, -200){3}
    \Vertex(379,-200){3}
    \Vertex(375,-207){3}
    \SetColor{Black}
    \Line(376,-195)(395,-165)
    \Line(373,-195)(354,-164)
    \Vertex(353,-161){9}
    \SetColor{White}
    \Vertex(353, -161){3}
    \SetColor{Black}
    \Vertex(395,-163){9}
    \SetColor{White}
    \Vertex(391, -163){3}
    \Vertex(399, -163){3}
    \SetColor{Black}
    \Line(353,-156)(353,-113)
    \Vertex(353,-111){9}
    \SetColor{White}
    \Vertex(350, -108){2.5}
    \Vertex(356, -108){2.5}
    \Vertex(350, -114){2.5}
    \Vertex(356, -114){2.5}
  \end{picture}
}}\,}

\begin{document}
%%%%%%%%%%%%%%%%%%%%%%%%%%%%%%%%%%%%%%%%%%%%%
%%%%%%%%%%%%%%%%%%%%%%%%%%%%%%%%%%%%%%%%%%%%%

\title[quasi-ormoulds]{The Hopf algebra of finite topologies and mould composition}

\author{Fr\'ed\'eric Fauvet}
\address{IRMA,
10 rue du G\'en\'eral Zimmer,
67084 Strasbourg Cedex, France}
	   \email{frederic.fauvet@gmail.com}
 	  \urladdr{}
	 
\author{Lo\"\i c Foissy}
\address{Universit\'e du Littoral - C\^ote d'Opale, Calais}
	   \email{Loic.Foissy@lmpa.univ-littoral.fr}
 	  \urladdr{}
         
\author{Dominique Manchon}
\address{Universit\'e Blaise Pascal,
         C.N.R.S.-UMR 6620, BP 80026,
         63171 Aubi\`ere, France}       
         \email{manchon@math.univ-bpclermont.fr}
         \urladdr{http://math.univ-bpclermont.fr/~manchon/}

%%%%%%%%%%%%%%%%%%%%%%%%%%%%%%%%%%%%%%%%%%%%%%%%%%%%%%%%%%%%%%%%%%%
\date{January 26, 2015}
%%%%%%%%%%%%%%%%%%%%%%%%%%%%%%%%%%%%%%%%%%%%%%%%%%%%%%%%%%%%%%%%%%%

\begin{abstract}
We exhibit an internal coproduct on the Hopf algebra of finite topologies recently defined by the second author, C. Malvenuto and F. Patras,
dual to the composition of "quasi-ormoulds", which are the natural version of J. Ecalle's moulds in this setting.
All these results are displayed in the linear species formalism.

\bigskip


\smallskip

\end{abstract}

%
%\begin{altabstract}
%\end{altabstract}
%

\maketitle

%%%%%%%%%%%%%%%%%%%%%%%%%%%%%%%%%%%%%%%%%%%%%%%%%%%%%%%%%%%%%%%%%%%

% \tableofcontents

%\newpage

%%%%%%%%%%%%%%%%%%%%%%%%%%%%%%%%%%%%%%%%%%%%%%%%%%%%%%%%%%%%%%%%%%%

\section{Introduction}
\label{sect:intro}
%%%%%

The study of finite topological spaces was initiated by Alexandroff in 1937,
and revived at several periods since then, using the natural bijection,
recalled below, which exists between these spaces and finite sets endowed with
a quasi--order. In \cite{FMP}, the topic was reexamined through the angle of Hopf
algebraic techniques, which have proved quite pervasive in algebraic
combinatorics in recent years. A number of  so--called combinatorial Hopf algebras
(graded and linearly spanned by combinatorial objects) are now of constant use
in many parts of mathematics, with  frequent occurences of the Hopf
algebras of shuffles and quasishuffles, non commutative symmetric
functions, Connes--Kreimer, Malvenuto--Reutenauer, word quasisymmetric
functions $\WQSym$, etc \cite{FNT14,FPT13,Gessel,Hoff,MR95}. This type of machinery to study finite spaces was
implemented in the article \cite{FMP}, with the introduction of a commutative Hopf
algebra $\mathcal{H}$ based on (isomorphism classes of) quasi--posets. These
constructions were investigated further in the article \cite{FM} , with in
particular the description of a non commutative and non cocommutative Hopf
algebra $\mathcal{H}_T$ based on labelled quasi--posets. In the present text
we show that both $\mathcal{H}$ and $\mathcal{H}_T$ can be endowed with a
second coproduct, which is degree--preserving and as such called internal.

The construction of the coproduct is non--trivial and is in fact achieved
within the formalism of linear species. It would certainly have been very
difficult to find by simple guess but it is in fact directly inspired by an
operation known in J. Ecalle's mould calculus (\cite{Eca92,EV04} as {\tmem{mould composition}}.
The basic facts on these combinatorial objects are recalled in the present text.

In \cite{FM}, a family of natural morphisms from \ $\mathcal{H}_T$ to $\WQSym$ was also
constructed, based on the classical concept of linear extensions (\cite{Sta11}). In
the present text, we show that one of these morphisms also respects the
internal coproduct. Once again, this is realized at the level of species, with
the introduction of a {\tmem{species of set compositions}}, which is a natural
framework to define a morphism which specializes to applications from the Hopf
algebras of quasi--posets onto $\QSym$ an $\WQSym$ (in the commutative and
non--commutative cases respectively) respecting the products and both the
external and internal products. Our results notably entail the definition of a
natural internal coproduct on $\WQSym$.

Recall (see e.g. \cite[\S 2.1]{FM}) that a topology on a finite set $X$ is given by the family $\Cal T$ of open subsets of $X$ subject to the three following axioms:
\begin{itemize}
\item $\emptyset\in\Cal T$, $X\in\Cal T$,
\item The union of a finite number of open subsets is an open subset,
\item The intersection of a finite number of open subsets is an open subset.
\end{itemize}
The finiteness of $X$ allows to consider only finite unions in the second axiom, so that axioms 2 and 3 become dual to each other. In particular the dual topology is defined by 
\begin{equation}
\overline {\Cal T}:=\{X\backslash Y,\,Y\in\Cal T\}.
\end{equation}
In other words, open subsets in $\Cal T$ are closed subsets in $\overline{\Cal T}$ and vice-versa. Any topology $\Cal T$ on $X$ defines a quasi-order (i.e. a reflexive transitive relation) denoted by $\le_{\Cal T}$ on $X$:
\begin{equation}
x\le_{\Cal T}y\Longleftrightarrow \hbox{ any open subset containing $x$ also contains $y$}.
\end{equation}
Conversely, any quasi-order $\le$ on $X$ defines a topology $\Cal T_{\le}$ given by its \textsl{final segments}, i.e. subsets $Y\subset X$ such that ($y\in Y\hbox{ and }y\le z)\Rightarrow z\in Y$. Both operations are inverse to each other: $\le_{\Cal T_{\le}}=\le$ and ${\Cal T}_{\le_{\Cal T}}=\Cal T$. Hence there is a natural bijection between topologies and quasi-orders on a finite set $X$.\\

\noindent Any quasi-order (hence any topology $\Cal T$) on $X$ gives rise to an equivalence class:
\begin{equation}
x\sim_{\Cal T}y\Longleftrightarrow (x\le_{\Cal T} y \hbox{ and }y\le_{\Cal T} x).
\end{equation}
This equivalence relation is trivial if and only if the quasi-order is a (partial) order, which is equivalent to the fact that the topology $\Cal T$ is $T_0$. Any topology $\Cal T$ on $X$ defines a $T_0$ topology on the quotient $X/\sim_{\Cal T}$, corresponding to the partial order induced by the quasi-order $\le_{\Cal T}$. Hence any finite topological set can be represented by the Hasse diagram of its $T_0$ quotient.
\vskip 4mm
$$\arbreccbig$$
\vskip 4mm
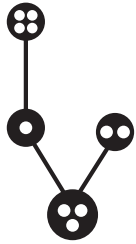
\centerline{\sl A finite topological space with 10 elements and 4 equivalence classes}
\vskip 4mm
\noindent\textbf{Acknowlegdements:} This work is supported by Agence Nationale de la Recherche, projet CARMA (Combinatoire Alg\'ebrique, R\'esurgence, Moules et Applications, ANR-12-BS01-0017).

%%%%%
\section{Refinement and quotient topologies}
%%%%%
Let $\Cal T$ and $\Cal T'$ be two topologies on a finite set $X$. We say that $\Cal T'$ is finer than $\Cal T$, and we write $\Cal T'\prec \Cal T$, when any open subset for $\Cal T$ is an open subset for $\Cal T'$. This is equivalent to the fact that for any $x,y\in X$, $x\le_{\Cal T'}y\Rightarrow x\le_{\Cal T}y$.\\

The \textsl{quotient} $\Cal T/\Cal T'$ of two topologies $\Cal T$ and $\Cal T'$ with $\Cal T'\prec \Cal T$ is defined as follows: the associated quasi-order $\le_{\Cal T/\Cal T'}$ is the transitive closure of the relation $\Cal R$ defined by:
\begin{equation}
x\Cal R y\Longleftrightarrow (x\le_{\Cal T} y\hbox{ or }y\le_{\Cal T'} x).
\end{equation}
Note that, contrarily to what is usually meant by "quotient topology", $\Cal T/\Cal T'$ is a topology on the same finite space $X$ than the one on which $\Cal T$ and $\Cal T'$ are given. The definitions immediately yield compatibility of the quotient with the involution, i.e.
\begin{equation}\label{quotient-inv}
\overline{\Cal T/\Cal T'}=\overline{\Cal T}\big /\,\overline{\Cal T'}.
\end{equation}
\noindent{\bf Examples:}
\begin{enumerate}
\item If $\Cal D$ is the discrete topology on $X$, for which any subset is open, the quasi-order $\le_{\Cal D}$ is nothing but $x\le_{\Cal D} y\Leftrightarrow x=y$, and then $\Cal T/\Cal D=\Cal T$.
\item For any topology $\Cal T$, the quotient $\Cal T/\Cal T$ has the same connected components than $\Cal T$, and the restriction of $\Cal T/\Cal T$ to any connected component is the coarse topology. In other words, for any $x,y\in X$, $x$ and $y$ are in the same connected component for $\Cal T$ if and only if $x\le_{\Cal T/\Cal T}y$, which is also equivalent to $x\sim_{\Cal T/\Cal T}y$.
\end{enumerate}
\begin{lem}\label{shrink}
Let $\Cal T''\prec\Cal T'\prec\Cal T$ be three topologies on $X$. Then $\Cal T'/\Cal T''\prec\Cal T/\Cal T''$, and we have the following equality between topologies on $X$:
\begin{equation}
\Cal T/\Cal T'=(\Cal T/\Cal T'')\Big/(\Cal T'/\Cal T'')
\end{equation}
\end{lem}
\begin{proof}
We compare the associated quasi-orders. The first assertion is obvious. For $x,y\in X$ we write $x\Cal R y$ for ($x\le_{\Cal T}y$ or $y\le_{\Cal T'}x$), and $x\Cal Q y$ for ($x\le_{\Cal T/\Cal T''}y$ or $y\le_{\Cal T'/\Cal T''}x$). We have $x\le_{\Cal T/\Cal T'}y$ if and only if there exist $a_1,\ldots, a_p\in X$ such that
$$x\Cal R a_1\Cal R\cdots\Cal R a_p\Cal R y.$$
On the other hand,
\begin{eqnarray*}
x\le_{(\Cal T/\Cal T'')\big/(\Cal T'/\Cal T'')}y &\Longleftrightarrow& \exists b_1,\ldots,b_q\in X,\, x\Cal Q b_1\Cal Q\cdots\Cal Q b_q\Cal Q y\\
&\Longleftrightarrow& \exists c_1,\ldots,c_r\in X,\, x\wt{\Cal R} c_1\wt{\Cal R}\cdots\wt{\Cal R} c_r\wt{\Cal R} y,
\end{eqnarray*}
with 
\begin{eqnarray*}
a\wt{\Cal R}b &\Longleftrightarrow& (a\le_{\Cal T}b \hbox{ or }b\le_{\Cal T''}a) \hbox{ or } (b\le_{\Cal T'}a \hbox{ or } a\le_{\Cal T''}b)\\
&\Longleftrightarrow& a\le_{\Cal T}b \hbox{ or } b\le_{\Cal T'}a\\
&\Longleftrightarrow& a\Cal R b.
\end{eqnarray*}
Hence,
$$x\le_{(\Cal T/\Cal T'')\big/(\Cal T'/\Cal T'')}y\Longleftrightarrow x\le_{\Cal T/\Cal T'} y.$$
\end{proof}
\begin{defn}
Let $\Cal T'\prec \Cal T$ be two topologies on $X$. We will say that $\Cal T'$ is \textsl{$\Cal T$-admissible} if
\begin{itemize}
\item $\Cal T'\restr{Y}=\Cal T\restr{Y}$ for any subset $Y\subset X$ connected for the topology $\Cal T'$, 
\item For any $x,y\in X$, $x\sim_{\Cal T/\Cal T'}y\Longleftrightarrow x\sim_{\Cal T'/\Cal T'}y$.
\end{itemize}
\end{defn}
In particular, $\Cal T$ is $\Cal T$-admissible. We write $\Cal T'\oprec \Cal T$ when $\Cal T'\prec \Cal T$ and $\Cal T'$ is $\Cal T$-admissible. Note that the reverse implication in the second axiom is always true for $\Cal T'\prec\Cal T$. It easily follows from \eqref{quotient-inv} that $\Cal T'\oprec \Cal T$  if and only if $\overline{\Cal T'}\oprec \overline{\Cal T}$ .
\begin{lem}\label{classes}
If $\Cal T'\oprec \Cal T$, then we have for any $x,y\in X$:
$$x\sim_{\Cal T'}y\Longleftrightarrow x\sim_{\Cal T}y.$$
\end{lem}
\begin{proof}
The direct implication is obvious. Conversely, if $x\sim_{\Cal T}y$ then $x\sim_{\Cal T/\Cal T'}y$, hence $x\sim_{\Cal T'/\Cal T'}y$, which means that $x$ and $y$ are in the same $\Cal T'$-connected component. The restrictions of $\Cal T$ and $\Cal T'$ on this component coincide, hence $x\sim_{\Cal T'}y$ .
\end{proof}
\begin{lem}\label{composantes-connexes}
If $\Cal T'\prec \Cal T$, the connected components of $\Cal T/\Cal T'$ are the same than those of $\Cal T$.
\end{lem}
\begin{proof}
The connected components of $\Cal T$, resp. $\Cal T/\Cal T'$, are nothing but the equivalence classes for $\Cal T/\Cal T$, resp.$( \Cal T/\Cal T')\big/(\Cal T/\Cal T')$. These two topologies coincide according to Lemma \ref{shrink}. 
\end{proof}
\begin{prop}\label{transitivite}
The relation $\oprec$ is transitive.
\end{prop}
\begin{proof}
Let $\Cal T''\prec\Cal T'\prec \Cal T$ be three topologies on $X$. Suppose that $\Cal T''$ is $\Cal T'$-admissible, and that $\Cal T'$ is $\Cal T$-admissible. If $Y\subset X$ is $\Cal T''$-connected, it is also $\Cal T'$-connected, hence $\Cal T''\restr Y=\Cal T'\restr Y=\Cal T\restr Y$. Now let $x,y\in X$ with $x\sim_{\Cal T/\Cal T''}y$. By definition of the transitive closure, there exist $a_1,\ldots, a_p$ and $
b_1,\ldots,b_p$ in $X$ such that 
$$x\le_{\Cal T} a_1,\,b_1\le_{\Cal T}a_2,\ldots, b_p\le_{\Cal T}y$$
and $a_i\ge_{\Cal T''} b_i$ for $i=1,\ldots ,p$. We also have $a_i\ge_{\Cal T'} b_i$ for $i=1,\ldots ,p$ because $\Cal T''\prec\Cal T'$. Hence,
$$x\sim_{\Cal T/\Cal T'}a_1\sim_{\Cal T/\Cal T'}b_1\sim_{\Cal T/\Cal T'}\cdots
\sim_{\Cal T/\Cal T'}a_p\sim_{\Cal T/\Cal T'}b_p\sim_{\Cal T/\Cal T'}y,$$
from which we get:
$$x\sim_{\Cal T'/\Cal T'}a_1\sim_{\Cal T'/\Cal T'}b_1\sim_{\Cal T'/\Cal T'}\cdots
\sim_{\Cal T'/\Cal T'}a_p\sim_{\Cal T'/\Cal T'}b_p\sim_{\Cal T'/\Cal T'}y,$$
hence $x$ and $y$ are in the same $\Cal T'$-connected component. Using that the restrictions of $\Cal T$ and $\Cal T'$ on this component coincide, we get $x\sim_{\Cal T'/\Cal T''}y$. From $\Cal T'\oprec\Cal T$ we get then $x\sim_{\Cal T''/\Cal T''}y$. This ends up the proof of Proposition \ref{transitivite}.
\end{proof}
\begin{lem}
If $\Cal T"\oprec\Cal T'\oprec\Cal T$, then $\Cal T'/\Cal T''\oprec\Cal T/\Cal T''$.
\end{lem}
\begin{proof}
Let $x,y\in X$ with $x\sim_{(\Cal T/\Cal T'')/(\Cal T'/\Cal T'')}y$. Then $x\sim_{\Cal T/\Cal T'}y$ according to Lemma \ref{shrink}, hence $x\sim_{\Cal T'/\Cal T'}y$, hence $x\sim_{(\Cal T'/\Cal T'')/(\Cal T'/\Cal T'')}y$ applying Lemma \ref{shrink} again.
\end{proof}
\begin{prop}\label{transit}
Let $\Cal T$ and $\Cal T''$ be two topologies on $X$. If $\Cal T''\oprec\Cal T$, then $\Cal T'\mapsto\Cal T'/\Cal T''$ is a bijection from the set of topologies $\Cal T'$ on $X$ such that $\Cal T''\oprec\Cal T$, onto the set of topologies $\Cal U$ on $X$ such that $\Cal U\oprec \Cal T/\Cal T''$.
\end{prop}
\begin{proof}
Given $\Cal U\oprec \Cal T/\Cal T''$, we have to prove the existence of a unique $\Cal T'$ such that $\Cal T''\oprec\Cal T'\oprec\Cal T$ and $\Cal U=\Cal T'/\Cal T''$. According to Lemma \ref{composantes-connexes}, the connected components of $\Cal T'$ must be those of $\Cal U$. The topologies $\Cal T'$ and $\Cal T$ must coincide on each of these components, which uniquely defines $\Cal T'$.\\

Let us now check $\Cal T''\oprec\Cal T'\oprec\Cal T$: if $x\le_{\Cal T'}y$, then $x$ and $y$ are in the same $\Cal T'$-connected component, on which $\Cal T$ and $\Cal T'$ coincide. Hence $x\le_{\Cal T}y$, which means $\Cal T'\prec \Cal T$. Now suppose $x\le_{\Cal T''}y$. Then $x\le_{\Cal T} y$, which implies $x\le_{\Cal T/\Cal T''}y$, which in turn implies $x\le_{(\Cal T/\Cal T'')/\Cal U}y$. The latter is equivalent to $x\le_{\Cal U/\Cal U}y$, as well as to $x\le_{\Cal T'/\Cal T'}y$. In other words, $x$ and $y$ are in the same $\Cal T'$-connected component. Moreover, since $x\le_{\Cal T}y$ we also have $x\le_{\Cal T'}y$ by definition of $\Cal T'$. This proves $\Cal T''\prec \Cal T'$.\\

If $x\le_{\Cal U}y$, it means that $x$ and $y$ are in the same $\Cal U$-connected component, and moreover $x\le_{\Cal T/\Cal T''}y$, because $\Cal U\oprec \Cal T/\Cal T''$. By definition of the transitive closure, there exist $a_1,\ldots, a_p$ and $
b_1,\ldots,b_p$ in $X$ such that 
\begin{equation}\label{TsurTsec}
x\le_{\Cal T} a_1,\,b_1\le_{\Cal T}a_2,\ldots, b_p\le_{\Cal T}y
\end{equation}
and $a_i\ge_{\Cal T''} b_i$ for $i=1,\ldots ,p$. In particular, $a_i\sim_{\Cal T/\Cal T''} b_i$, hence:
$$x\sim_{\Cal T/\Cal T''} a_1\sim_{\Cal T/\Cal T''}b_1\sim_{\Cal T/\Cal T''}a_2\sim_{\Cal T/\Cal T''}\cdots\sim_{\Cal T/\Cal T''}b_p\sim_{\Cal T/\Cal T''}y$$
which immediately yields:
$$x\sim_{(\Cal T/\Cal T'')/\Cal U} a_1\sim_{(\Cal T/\Cal T'')/\Cal U}b_1\sim_{(\Cal T/\Cal T'')/\Cal U}a_2\sim_{(\Cal T/\Cal T'')/\Cal U}\cdots\sim_{(\Cal T/\Cal T'')/\Cal U}b_p\sim_{(\Cal T/\Cal T'')/\Cal U}y$$
since $\Cal T/\Cal T''\prec (\Cal T/\Cal T'')/\Cal U$. Now using $\Cal U\oprec\Cal T/\Cal T''$ again, we get
$$x\sim_{\Cal U/\Cal U} a_1\sim_{\Cal U/\Cal U}b_1\sim_{\Cal U/\Cal U}a_2\sim_{\Cal U/\Cal U}\cdots\sim_{\Cal U/\Cal U}b_p\sim_{\Cal U/\Cal U}y.$$
Hence all the chain is included in the same $\Cal U$-connected component. By definition of $\Cal T'$ we can then rewrite \eqref{TsurTsec} as:
\begin{equation}\label{TprimesurTsec}
x\le_{\Cal T'} a_1,\,b_1\le_{\Cal T'}a_2,\ldots, b_p\le_{\Cal T'}y
\end{equation}
with $a_i\ge_{\Cal T''} b_i$ for $i=1,\ldots ,p$, which means $x\le_{\Cal T'/\Cal T''}y$.\\

Conversely, if $x\le_{\Cal T'/\Cal T''}y$, then $x$ and $y$ are in the same $\Cal U$-component according to the definition of $\Cal T'$, and \eqref{TprimesurTsec} implies \eqref{TsurTsec}. Hence $x\le_{\Cal T/\Cal T''}y$, hence $x\le_{\Cal U}y$. We have then:
\begin{equation}\label{uttprime}
\Cal U=\Cal T/\Cal T'.
\end{equation}
To finish the proof, we have to show $\Cal T'\oprec\Cal T$ and $\Cal T''\oprec\Cal T'$. Any $\Cal T'$-connected subset $Y\subset X$ is also $\Cal U$-connected,  hence the restrictions of $\Cal T$ and $\Cal T'$ on $Y$ coincide. Similarly, the restrictions of $\Cal T'$ and $\Cal T''$ on any $\Cal T''$-connected subset coincide. If $x\sim_{\Cal T/\Cal T'}y$, then $x\sim_{(\Cal T/\Cal T'')/(\Cal T'/\Cal T'')}y$, which means $x\sim_{(\Cal T/\Cal T'')/\Cal U}y$, which in turn yields $x\sim_{\Cal U/\Cal U}y$, i.e. $x\sim_{\Cal T'/\Cal T'}y$. Hence $\Cal T'\oprec \Cal T$. Finally, if $x\sim_{\Cal T'/\Cal T''}y$, then $x\sim_{\Cal T/\Cal T''}y$, hence $x\sim_{\Cal T''/\Cal T''}y$, which yields $\Cal T''\oprec \Cal T'$. This ends up the proof of Proposition \ref{transit}.
\end{proof}
%%%%%
\section{Algebraic structures on finite topologies}\label{sect:alg}
%%%%%
The collection of all finite topological spaces shows very rich algebraic features, best viewed in the linear species formalism. We  describe a commutative product, an "internal" coproduct and an "external" coproduct, as well as the interactions between them.
%%%
\subsection{The coalgebra species of finite topological spaces}\label{sect:espece-cogebre}
%%%
Recall that a linear species is a contravariant functor from the category of finite sets with bijections into the category of vector spaces (on some field $K$). The species $\mathbb T$ of topological spaces is defined as follows: $\mathbb T_X$ is the vector space freely generated by the topologies on $X$. For any bijection $\varphi:X\longrightarrow X'$, the isomorphism $\mathbb T_{\varphi}:\mathbb T_{X'}\longrightarrow \mathbb T_X$ is defined by the obvious relabelling:
$$\mathbb T_{\varphi}(\Cal T):=\{\varphi^{-1}(Y),\,Y\in\Cal T\}$$
for any topology $\Cal T$ on $X'$. For any finite set $X$, let us introduce the coproduct $\Gamma$ on $\mathbb T_X$ defined as follows:
\begin{equation}
\Gamma(\Cal T)=\sum_{{\Cal T'}\soprec{\Cal T}}\Cal T'\otimes \Cal T/\Cal T'.
\end{equation}

{\bf Examples.} If $X=E\sqcup F=A\sqcup A\sqcup C$ are two partitions of $X$:
\begin{align*}
\Gamma(\tdun{$X$})&=\tdun{$X$}\otimes \tdun{$X$},\\
\Gamma(\tddeux{$E$}{$F$})&=\tddeux{$E$}{$F$}\otimes \tdun{$X$}+\tdun{$E$}\tdun{$F$}\otimes \tddeux{$E$}{$F$}\\
\Gamma(\tdun{$E$}\tdun{$F$})&=\tdun{$E$}\tdun{$F$}\otimes \tdun{$E$}\tdun{$F$}\\
\Gamma(\tdtroisun{$A$}{$C$}{$B$})&=\tdtroisun{$A$}{$C$}{$B$}\otimes \tdun{$X$}+\tddeux{$A$}{$B$}\tdun{$C$}\otimes \tddeux{$A\sqcup B$}{$C$}
\hspace{.5cm}+\tddeux{$A$}{$C$}\tdun{$B$}\otimes \tddeux{$A\cup C$}{$B$}\hspace{.5cm}
+\tdun{$A$}\tdun{$B$}\tdun{$C$}\otimes \tdtroisun{$A$}{$C$}{$B$}\\
\Gamma(\tdtroisdeux{$A$}{$B$}{$C$})&=\tdtroisdeux{$A$}{$B$}{$C$}\otimes \tdun{$X$}+\tddeux{$A$}{$B$}\tdun{$C$}\otimes \tddeux{$A\sqcup B$}{$C$}
\hspace{.5cm}+\tdun{$A$}\tddeux{$B$}{$C$}\otimes \tddeux{$A$}{$B\sqcup C$}\hspace{.5cm}+\tdun{$A$}\tdun{$B$}\tdun{$B$}\otimes 
\tdtroisdeux{$A$}{$B$}{$C$}\\
\Gamma(\pdtroisun{$A$}{$B$}{$C$})&=\pdtroisun{$A$}{$B$}{$C$}\otimes \tdun{$X$}+\tddeux{$B$}{$A$}\tdun{$C$}\otimes \tddeux{$C$}{$A\sqcup B$}
\hspace{.5cm}+\tddeux{$C$}{$A$}\tdun{$B$}\otimes \tddeux{$B$}{$A\sqcup C$}\hspace{.5cm}
+\tdun{$A$}\tdun{$B$}\tdun{$C$}\otimes \pdtroisun{$A$}{$B$}{$C$}\\
\Gamma(\tddeux{$A$}{$B$}\tdun{$C$})&=\tddeux{$A$}{$B$}\tdun{$C$}\otimes \tdun{$A\sqcup B$}\hspace{.5cm}\tdun{$C$}
+\tdun{$A$}\tdun{$B$}\tdun{$C$}\otimes \tddeux{$A$}{$B$}\tdun{$C$}\\
\Gamma(\tdun{$A$}\tdun{$B$}\tdun{$C$})&=\tdun{$A$}\tdun{$B$}\tdun{$C$}\otimes \tdun{$A$}\tdun{$B$}\tdun{$C$}
\end{align*}

\begin{thm}\label{main}
The coproduct $\Gamma$ is coassociative.
\end{thm}
\begin{proof}
For any topology $\Cal T$ on $X$ we have:
\begin{equation}
(\Gamma\otimes\mop{Id})\Gamma(\Cal T)=\sum_{\Cal T''\soprec\Cal T'\soprec\Cal T}\Cal T''\otimes \Cal T'/\Cal T''\otimes\Cal T/\Cal T',
\end{equation}
whereas
\begin{equation}
(\mop{Id}\otimes\Gamma)\Gamma(\Cal T)=\sum_{\Cal T''\soprec\Cal T'}\hskip 3mm \sum_{\Cal U\soprec \Cal T'/\Cal T''}\Cal T''\otimes\Cal U\otimes(\Cal T/\Cal T'')\big/\Cal U.
\end{equation}
The result then comes from Lemmas \ref{transitivite} and \ref{shrink}, and from Proposition \ref{transit}.
\end{proof}
The group-like elements of $\mathbb T_ X$ are the topologies $\Cal T$ such that for any connected component $Y$ of $\Cal T$,
$\Cal T_{\mid Y}$ is coarse: in, other words, $\Cal T$ is group-like if, and only if, $\leq_{\Cal T}$ is an equivalence.
For any topology $\Cal T$ on $X$, ther exists a unique group-like topology $\Cal T'\oprec \Cal T$, namely
the group-like topology $\Cal T'$ such that $\leq_{\Cal T'}=\sim_\Cal T$; moreover, $\Cal T/\Cal T'=\Cal T$.
The unique topology $\Cal T''$ such that $\Cal T/\Cal T''$ is group-like is $\Cal T''=\Cal T$. Hence, linear form $\varepsilon_X$ on $\mathbb T_X$ 
defined by $\varepsilon_X(\Cal T)=1$ if $\Cal T$ is group-like and $\varepsilon(\Cal T)=0$ otherwise is a counit.

The involution $\Cal T\mapsto \overline{\Cal T}$ obviously extends linearly to a coalgebra involution on $\mathbb T_X$. Any relabelling induces an involutive coalgebra isomorphism in a functorial way. To summarize:
\begin{cor}
$\mathbb T$ is a species is the category of counital connected coalgebras with involution.
\end{cor}
A commutative associative product on finite topologies is defined as follows: for any pair $X_1,X_2$ of finite sets we introduce 
\begin{eqnarray*}
m:\mathbb T_{X_1}\otimes\mathbb T_{X_2}&\longrightarrow& \mathbb T_{X_1\sqcup X_2}\\
\Cal T_1\otimes\Cal T_2 &\longmapsto \Cal T_1\Cal T_2,
\end{eqnarray*}
where $\Cal T_1\Cal T_1$ is characterized by $Y\in\Cal T_1\Cal T_2$ if and only if $Y\cap X_1\in\Cal T_1$ and $Y\cap X_2\in\Cal T_2$.
\begin{prop}\label{mult-un}
The species coproduct $\Gamma$ and the product are compatible, i.e. for any pair $X_1,X_2$ of finite sets the following diagram commutes:
\diagramme{
\xymatrix{
\mathbb T_{X_1}\otimes\mathbb T_{X_2}\ar [rrrr]^m\ar[d]_{\Gamma\otimes\Gamma}
&&&& \mathbb T_{X_1\sqcup X_2}\ar[d]^{\Gamma}\\
\mathbb T_{X_1}\otimes\mathbb T_{X_1}\otimes\mathbb T_{X_2}\otimes\mathbb T_{X_2}
\ar[drr]_{\tau^{2,3}} &&&&\mathbb T_{X_1\sqcup X_2}\otimes\mathbb T_{X_1\sqcup X_2}\\
&& \mathbb T_{X_1}\otimes\mathbb T_{X_2}\otimes\mathbb T_{X_1}\otimes\mathbb T_{X_2}
\ar[urr]_{m\otimes m}&&
}
}
\end{prop}
\begin{proof}
Let $\Cal T_1$, resp. $\Cal T_2$ be a topology on $X_1$, resp. $X_2$. Let $\Cal U_1\oprec\Cal T_1$ and $\Cal U_2\oprec\Cal T_2$. Then $\Cal U_1\Cal U_2\oprec\Cal T_1\Cal T_2$. Conversely, any topology $\Cal U$ on $X_1\sqcup X_2$ such that $\Cal U\oprec\Cal T_1\Cal T_2$ can be written $\Cal U_1\Cal U_2$ with $\Cal U_i=\Cal U\restr{X_i}$ for $i=1,2$, and we have $\Cal U_i\oprec\Cal T_i$. We have then:
\begin{eqnarray*}
\Gamma(\Cal T_1\Cal T_2)&=&\sum_{\Cal U\soprec\Cal T_1\Cal T_2}\Cal U\otimes (\Cal T_1\Cal T_2)/\Cal U\\
&=&\sum_{{\scriptstyle\Cal U_1\soprec\Cal T_1\atop \scriptstyle\Cal U_2\soprec\Cal T_2}}
\Cal U_1\Cal U_2\otimes(\Cal T_1/\Cal U_1)(\Cal T_2/\Cal U_2)\\
&=&\Gamma(\Cal T_1)\Gamma(\Cal T_2).
\end{eqnarray*}
\end{proof}
Finally, recall that the group-like elements in $\mathbb T_X$ are precisely the topologies $\Cal T_\Cal P$  where $\Cal P$ is a partition of $X$, defined as the product of the coarse topologies on each block of $\Cal P$. This suggests a grading on $\mathbb T_X$: we introduce $d(\Cal T)$ as the number of equivalence classes minus the number of connected components of $\Cal T$. It is easy to see that this grading makes $(\mathbb T_X, \Gamma)$ a finite-dimensional graded coalgebra. The degree zero topologies are the group-like ones, i.e. the products of coarse topologies described above, and the maximum possible degree $|X|-1$ is reached for connected $\Cal T_0$ topologies.
%%%
\subsection{The external coproduct}
%%%
For any topology $\Cal T$ on a finite set $X$ and for any subset $Y\subset X$, we denote by $\Cal T\restr{Y}$ the restriction of $\Cal T$ to $Y$. It is defined by:
$$\Cal T\restr{Y}=\{Z\cap Y,\,Z\in\Cal T\}.$$
Restriction and taking quotients commute: for any subset $Y\subset X$ and for any $\Cal T'\oprec\Cal T$ we have $\Cal T'\restr{Y}\oprec\Cal T\restr{Y}$ and:
\begin{equation}\label{restr-quotient}
(\Cal T/\Cal T')\restr{Y}=\Cal T\restr Y\big/\Cal T'\restr Y.
\end{equation}
The external coproduct is defined on $\mathbb T_X$ as follows:
\begin{eqnarray*}
\Delta:\mathbb T_X&\longrightarrow& \bigoplus_{Y\subset X}\mathbb T_{X\setminus Y}\otimes\mathbb T_{Y}\\
\Cal T&\longmapsto& \sum_{Y\in\Cal T}\Cal T\restr{X\setminus Y}\otimes\Cal T\restr{Y}.
\end{eqnarray*}
\begin{prop}
The external coproduct is coassociative and multiplicative, i.e. the two following diagrams commute:
\diagramme{
\xymatrix{
\mathbb T_{X}\ar [rrr]^\Delta\ar[d]_{\Delta}
&&& \displaystyle\bigoplus_{Y\subset X}\mathbb T_{X\setminus Y}\otimes\mathbb T_{Y} \ar[d]^{I\otimes\Delta}\\
\displaystyle\bigoplus_{Z\subset X}\mathbb T_{X\setminus Z}\otimes\mathbb T_{Z}
\ar[rrr]_{\Delta\otimes I} &&&\displaystyle\bigoplus_{Z\subset Y\subset X}\mathbb T_{X\setminus Y}\otimes\mathbb T_{Y\setminus Z}\otimes\mathbb T_Z
}
}
and
\diagramme{
\xymatrix{
\mathbb T_{X_1}\otimes\mathbb T_{X_2}\ar [rr]^m\ar[d]_{\Delta\otimes\Delta}
&& \mathbb T_{X_1\sqcup X_2}\ar[d]^{\Delta}\\
\displaystyle\bigoplus_{{\scriptstyle Y_1\subset X_1\atop \scriptstyle Y_2\subset X_2}}
\mathbb T_{X_1\setminus Y_1}\otimes\mathbb T_{Y_1}\otimes\mathbb T_{X_2\setminus Y_2}\otimes\mathbb T_{Y_2}
\ar[dr]_{\tau^{2,3}} &&\displaystyle\bigoplus_{Y\subset X_1\sqcup X_2}\mathbb T_{(X_1\sqcup X_2)\setminus Y}\otimes\mathbb T_{Y}\\
& \displaystyle\bigoplus_{{\scriptstyle Y_1\subset X_1\atop \scriptstyle Y_2\subset X_2}}
\mathbb T_{X_1\setminus Y_1}\otimes\mathbb T_{X_2\setminus Y_2}\otimes\mathbb T_{Y_1}\otimes\mathbb T_{Y_2}
\ar[ur]_{m\otimes m}&
}
}
\end{prop}
\begin{proof}
we have:
\begin{equation}
(\Delta\otimes I)\Delta(\Cal T)=\bigoplus_{Z\in\Cal T,\, \wt Y\in\Cal T\srestr{X\setminus Z}}\Cal T\restr{X\setminus{Z\sqcup \wt Y}}\otimes\Cal T\restr{\wt Y}\otimes \Cal T\restr{Z}
\end{equation}
and
\begin{equation}
(I\otimes\Delta)\Delta(\Cal T)=\bigoplus_{Y,Z\in\Cal T,\, Z\subset Y}\Cal T\restr{X\setminus Y}\otimes\Cal T\restr{Y\setminus Z}\otimes \Cal T\restr{Z}
\end{equation}
Coassociativity then comes from the obvious fact that $(\wt Y,Z)\mapsto\wt Y\sqcup Z$ is a bijection from the set of pairs $(\wt Y,Z)$ with $Z\in\Cal T$ and $\wt Y\in\Cal T\restr{X\setminus Z}$, onto the set of pairs $(Y,Z)$ of elements of $\Cal T$ subject to $Z\subset Y$. The inverse map is given by $(Y,Z)\mapsto (Y\cap X\setminus Z,Z)$. The multiplicativity property $\Delta(\Cal T_1\Cal T_2)=\Delta(\Cal T_1)\Delta(\Cal T_2)$ comes straightforwardly from the very definition of the topology $\Cal T_1\Cal T_2$ on the disjoint union $X_1\sqcup X_2$.
\end{proof}
\begin{thm}
The internal and external coproducts are compatible, in the sense that the following diagram commutes for any finite set $X$:
\diagramme{
\xymatrix{
\mathbb T_X\ar[rr]^{\Gamma}\ar[d]_{\Delta} &&
\mathbb T_X\otimes\mathbb T_X\ar[d]^{I\otimes\Delta}\\
\displaystyle\bigoplus_{Y\subset X}\mathbb T_{X\setminus Y}\otimes \mathbb T_Y\ar[dr]_{\Gamma\otimes\Gamma}&&
\displaystyle\bigoplus_{Y\subset X}\mathbb T_X\otimes\mathbb T_{X\setminus Y}\otimes \mathbb T_Y\\
&\displaystyle\bigoplus_{Y\subset X}\mathbb T_{X\setminus Y}\otimes \mathbb T_{X\setminus Y}\otimes\mathbb T_{Y}\otimes \mathbb T_Y\ar[ur]_{m^{1,3}}&
}
}
\end{thm}
\begin{proof}
For any $\Cal T\in\mathbb T_X$ we have:
\begin{eqnarray}
(I\otimes\Delta)\circ\Gamma(\Cal T)&=&(I\otimes\Delta)\sum_{\Cal U\soprec\Cal T}\Cal T\otimes\Cal T/\Cal U\nonumber\\
&=&\sum_{\Cal U\soprec\Cal T}\sum_{Y\in\Cal T/\Cal U}\Cal U\otimes(\Cal T/\Cal U)\restr{X\setminus Y}\otimes (\Cal T/\Cal U)\restr{Y}\nonumber\\
&=&\sum_{\Cal U\soprec\Cal T}\sum_{Y\in\Cal T/\Cal U}\Cal U\otimes\Cal T\restr{X\setminus Y}\Big/\Cal U\restr{X\setminus Y}\otimes\Cal T\restr{Y}\Big/\Cal U\restr{Y},\label{lhs}
\end{eqnarray}
whereas
\begin{eqnarray}
m^{1,3}\circ(\Gamma\otimes\Gamma)\circ\Delta(\Cal T)&=&m^{1,3}\circ(\Gamma\otimes\Gamma)
\sum_{Z\in\Cal T}\Cal T\restr{X\setminus Z}\otimes\Cal T\restr{Z}\nonumber\\
&=&\sum_{Z\in\Cal T}\ 
\sum_{{\scriptstyle \Cal U_1\soprec\Cal T\restr{X\setminus Z}\atop\scriptstyle \Cal U_2\soprec\Cal T\restr{Z}}}
\Cal U_1\Cal U_2\otimes \Cal T\restr{X\setminus Z}\Big/\Cal U_1\otimes\Cal T\restr{Z}\Big/\Cal U_2.\label{rhs}
\end{eqnarray}
Now, $Y\in\Cal T/\Cal U$ means that $Y$ is a final segment for $\le_{\Cal T/\Cal U}$, i.e. for any $y\in Y$, if $z\le_{\Cal T/\Cal U}y$, then $z\in Y$. A fortiori $z\in Y$ if $z\le_{\Cal U}y$ or $y\le_{\Cal U}z$. Then $Y$ is both a final and initial segment for $\le_{\Cal U}$, i.e. both closed and open for $\Cal U$, which yields $\Cal U=\Cal U_1\Cal U_2$, with $\Cal U_1=\Cal U\restr{X\setminus Y}$ and $\Cal U_2=\Cal U\restr{Y}$.\\

Conversely, if $\Cal U=\Cal U\restr{X\setminus Y}\Cal U\restr{Y}$, then for $y\in Y$ and any $z\in X$ such that $y\le_{\Cal U}z$ or $z\le_{\Cal U}y$, we have $z\in Y$. By iteration we have $y\le_{\Cal U/\Cal U}z\Rightarrow z\in Y$. But $\Cal U\oprec\Cal T$, hence $y\le_{\Cal T/\Cal U}z\Rightarrow z\in Y$, which means $Y\in\Cal T/\Cal U$. This proves that \eqref{lhs} and \eqref{rhs} coincide.
\end{proof}
%%%%%%
\section{Two commutative bialgebra structures}
%%%%%%
Consider the graded vector space:
\begin{equation}
\Cal H=\bigoplus_{n\ge 0}\Cal H_n,
\end{equation}
where $\Cal H_0=k.\un$, and where $\Cal H_n$ is the linear span of topologies on $\{1,\ldots,n\}$ when $n\ge 1$, modulo homeomorphisms. It can be seen as the quotient of the species $\mathbb T$ by the "forget the labels" equivalence relation: $\Cal T\sim\Cal T'$ if $\Cal T$ (resp. $\Cal T'$) is a topology on a finite set $X$ (resp. $X'$), such that there is a bijection from $X$ onto $X'$ which is a homeomorphism with respect to both topologies. This equivalence relation is compatible with the product and both coproducts introduced in Section \ref{sect:alg}, giving rise to a product $\cdot$ and two coproducts $\Gamma$ and $\Delta$ on $\Cal H$, the first coproduct being internal to each $\Cal H_n$. It naturally leads to the following:
\begin{thm}\label{structures}
The graded vector space $\Cal H$ is endowed with the following algebraic structures:
\begin{itemize}
\item $(\Cal H,\cdot,\Delta)$ is a commutative graded connected Hopf algebra.
\item $(\Cal H,\cdot,\Gamma)$ is a commutative bialgebra, graded by the degree $d$ introduced at the end of \S\ \ref{sect:espece-cogebre}.
\item $(\Cal H,\cdot,\Delta)$ is a comodule-coalgebra on $(\Cal H,\cdot,\Gamma)$. More precisely the following diagram of unital algebra morphisms commutes:
\diagramme{
\xymatrix{
\Cal H\ar[rr]^{\Gamma}\ar[d]_{\Delta} &&
\Cal H\otimes\Cal H\ar[d]^{I\otimes\Delta}\\
\Cal H\otimes\Cal H\ar[dr]_{\Gamma\otimes\Gamma}&&
\Cal H\otimes\Cal H\otimes\Cal H\\
&\Cal H\otimes\Cal H\otimes\Cal H\otimes \Cal H\ar[ur]_{m^{1,3}}&
}
}
\end{itemize}
\end{thm}
\begin{rmk}
The Hopf algebra of finite topologies of \cite{FM} is closely related, but the product is noncommutative due to renumbering. In fact, $\hbox{\bf T}_n$ stands for the set of topologies on $[n]=\{1,\ldots,n\}$, and $\hbox{\bf T}$ is the (disjoint) union of the $\hbox{\bf T}_n$'s for $n\ge 0$. For $\Cal T\in \hbox{\bf T}_n$ and $\Cal T'\in \hbox{\bf T}_{n'}$, the product $\Cal T\Cal T'$ is the topology on $[n+n']$ the open sets of which are $Y\sqcup (Y'+n)$, where $Y\in\Cal T$ and $Y'\in\Cal T'$. The two topologies $\Cal T\Cal T'$ and $\Cal T'\Cal T$ are not equal, though homeomeorphic. The "joint" product $\downarrow$, for which the open sets of $\Cal T\downarrow\Cal T'$ are the open sets $Y'$ of $\Cal T'$ and the sets $Y\sqcup \{n+1,\ldots,n+n'\}$ with $Y\in\Cal T$, is also associative. The empty set $\emptyset$ is the common unit for both products.
\end{rmk}
For any totally ordered finite set $E$ of cardinality $n$, let us denote by $\mop{Std}:E\to[n]$ the standardization map, i.e. the unique increasing bijection from $E$ onto $[n]$. This map yields a bijection form $\Cal P(E)$ onto $\Cal P([n])$ also denoted by $\mop{Std}$. The coproduct is defined by:
\begin{equation}
\Delta(\Cal T)=\sum_{Y\in \Cal T}\mop{Std}(\Cal T\restr{[n]\setminus Y})\otimes\mop{Std}(\Cal T\restr Y).
\end{equation}
\begin{prop}[\cite{FM} Proposition 6]
Let $\Cal H_{\hbox{\sevenbf T}}$ be the graded vector space freely generated by the $\hbox{\bf T}_n$'s. Then\\
\begin{enumerate}
\item $(\Cal H_{\hbox{\sevenbf T}},\cdot,\Delta)$ is a graded Hopf algebra,
\item $(\Cal H_{\hbox{\sevenbf T}},\downarrow,\Delta)$ is a graded infinitesimal Hopf algebra,
\item The involution $\Cal T\mapsto\overline{\Cal T}$ is a morphism for the product $\cdot$ and an antimorphism for the coproduct $\Delta$.
\end{enumerate}
\end{prop}
The internal coproduct $\Gamma$ on each homogeneous component of $\Cal H_{\hbox{\sevenbf T}}$ does not interact so nicely with the external coproduct $\Delta$ as it does in the commutative setting because of the shift and the standardization. 
Here is an example:
\begin{align*}
m^{1,3}\circ (\Gamma\otimes \Gamma)\circ \Delta(\tddeux{$3$}{$1,2$}\hspace{2mm})&=
\tddeux{$3$}{$1,2$}\hspace{2mm}\otimes \tdun{$1,2,3$}\hspace{.4cm}\otimes 1
+\tdun{$1,2$}\hspace{.2cm}\tdun{$3$}\otimes \tddeux{$3$}{$1,2$}\hspace{2mm}\otimes 1
+\tddeux{$3$}{$1,2$}\hspace{2mm}\otimes 1\otimes \tdun{$1,2,3$}\hspace{4mm}\\
&+\tdun{$1,2$}\hspace{2mm}\tdun{$3$}\otimes 1\otimes \tddeux{$3$}{$1,2$}\hspace{2mm}
+\textcolor{red}{\tdun{$1$}\tdun{$2,3$}\hspace{2mm} }\otimes \tdun{$1$}\otimes \tdun{$1,2$}\hspace{2mm},\\
(Id\otimes \Delta)\circ \Gamma(\tddeux{$3$}{$1,2$}\hspace{2mm})&=
\tddeux{$3$}{$1,2$}\hspace{2mm}\otimes \tdun{$1,2,3$}\hspace{.4cm}\otimes 1
+\tdun{$1,2$}\hspace{.2cm}\tdun{$3$}\otimes \tddeux{$3$}{$1,2$}\hspace{2mm}\otimes 1
+\tddeux{$3$}{$1,2$}\hspace{2mm}\otimes 1\otimes \tdun{$1,2,3$}\hspace{4mm}\\
&+\tdun{$1,2$}\hspace{2mm}\tdun{$3$}\otimes 1\otimes \tddeux{$3$}{$1,2$}\hspace{2mm}
+\textcolor{red}{\tdun{$1,2$}\hspace{2mm} \tdun{$3$}}\otimes \tdun{$1$}\otimes \tdun{$1,2$}\hspace{2mm}.
\end{align*}

%%%%%

%\newcommand{\setcomp}{\mathbb{SC}}
%\newcommand{\T}{\mathcal{T}}
%\newcommand{\topo}{\mathbb{T}}
%\newcommand{\lin}{\mathcal{L}}
%\newcommand{\h}{\mathcal{H}}
%\newcommand{\QSym}{\mathbf{QSym}}
%\newcommand{\WQSym}{\mathbf{WQSym}}

\section{Linear extensions and set compositions}

\subsection{Two Hopf algebras on words}

Let us first recall some facts on two well-known Hopf algebras.
Let $X$ be a totally ordered alphabet, and let $A=\mathbb{Q}[[X]]$ be the algebra of formal series generated by $X$.
A formal series $f \in A$ is quasi-symmetric if for any $X_1<\ldots<X_k$ and $Y_1<\ldots<Y_k$ in $X$, for any $a_1,\ldots,a_k\geq 1$,
the coefficients of $X_1^{a_1}\ldots X_k^{a_k}$ and of $Y_1^{a_1}\ldots Y_k^{a_k}$ in $f$ are equal.
The subalgebra of quasi-symmetric functions on $X$ will be denoted by $\QSym(X)$. For any composition $(a_1,\ldots,a_k)$, we put:
$$M_{(a_1,\ldots,a_k)}(X)=\sum_{X_1<\ldots<X_k} X_1^{a_1}\ldots X_k^{a_k}.$$
The family $(M_c(X))$ indexed by compositions linearly spends $\QSym(X)$; if $X$ is infinite, this is a basis. 

We shall use the following notation: if $i_1,\ldots,i_p\geq 0$, $\mop{QSh}(i_1,\ldots,i_p)$ is the set of surjections $u:[i_1+\ldots+i_p]\longrightarrow \hskip -5mm\longrightarrow[\max(u)]$,
such that for all $1\leq j\leq p$:
$$u_{i_1+\ldots+i_{j-1}+1}<\ldots<u_{i_1+\ldots+i_j}.$$
With this notation, for all compositions $(c_1,\ldots,c_k)$, $(c_{k+1},\ldots,c_{k+l})$:
$$M_{(c_1,\ldots,c_k)}(X)M_{(c_{k+1},\ldots,c_{k+l})}(X)
=\sum_{\sigma \in \smop{QSh}(k,l)} M_{(\sum_{\sigma(i)=1} c_i,\ldots, \sum_{\sigma(i)=\max(\sigma)}c_i)}(X).$$

Let $X,Y$ be two totally ordered alphabets.
\begin{enumerate}
\item  $X\sqcup Y$ is also totally ordered, the elements of $X$ being smaller than the elements of $Y$.
For any composition $(c_1,\ldots,c_k)$:
$$M_{(c_1,\ldots,c_k)}(X\sqcup Y)=\sum_{i=0}^k M_{(c_1,\ldots,c_i)}(X)M_{(c_{i+1},\ldots,c_k)}(Y).$$
\item $X\times Y$ is totally ordered by the lexicographic order. For any composition $(c_1,\ldots,c_k)$:
\begin{equation*}
M_{(c_1,\ldots,c_k)}(X\times Y)=\sum_{i_1+\cdots+i_p=k}
M_{(c_1,\ldots,c_{i_1})}(Y)\ldots M_{(c_{i_1+\cdots+i_{p-1}+1},\ldots,c_{i_1+\cdots+i_p})}(Y) M_{C_1,\ldots ,C_p}(X).
\end{equation*}
\end{enumerate}
with $C_1=c_1+\cdots+c_{i_1}$, \ldots , $C_p=c_{i_1+\cdots+i_{p-1}+1}+\cdots +c_{i_1+\cdots+i_p}$. Taking two denumerable infinite alphabets $X$ and $Y$, we identify $\QSym(X)$ and $\QSym(Y)$,
$\QSym(X)\otimes \QSym(Y)$ with $\QSym(X\sqcup Y)$ and $\QSym(X\times Y)$, $(x,y)$ being identified with $xy$,
and  we obtain a Hopf algebra $\QSym$, with a basis $(M_c)$ indexed by compositions. 
If $(c_1,\ldots,c_k)$ and $(c_{k+1},\ldots,c_{k+l})$ are compositions:
\begin{align*}
M_{(c_1,\ldots,c_k)}M_{(c_{k+1},\ldots,c_{k+l})}&=\sum_{\sigma \in \smop{QSh}(k,l)} M_{(\sum_{\sigma(i)=1} c_i,\ldots, \sum_{\sigma(i)=\max(\sigma)}c_i)},\\
\Delta(M_{(c_1,\ldots,c_k)})&=\sum_{i=0}^k M_{(c_1,\ldots,c_i)}\otimes M_{(c_{i+1},\ldots,c_k)},\\
\rho(M_{(c_1,\ldots,c_k)})&=\sum_{i_1+\cdots+i_p=k}
M_{(c_1,\ldots,c_{i_1})}\ldots M_{(c_{i_1+\cdots+i_{p-1}+1},\ldots,c_{i_1+\cdots+i_p})} \otimes M_{(C_1,\ldots,C_p)}.
\end{align*}

The construction of the Hopf algebra of packed words $\WQSym$ is similar. We now work in $B=\mathbb{Q}\langle\langle X\rangle\rangle$,
the algebra of noncommutative formal series generated by $X$. Recall that a packed word is a surjective map $w:[k]\longrightarrow\hskip -5mm\longrightarrow  [\max(w)]$, 
which we write as the word $w=w_1\ldots w_k$. Let $X_1\ldots X_k$ be a monomial in $B$. There is a unique bijective, increasing map $f$,
from $\{X_1,\ldots,X_k\}$ to a set $[m]$. Then $\mop{Pack}(X_1\ldots X_k)$ is the packed word $f(X_1)\ldots f(X_k)$.
For any packed word $w$, we put:
$$M_w(X)=\sum_{\smop{Pack}(X_1\ldots X_k)=w} X_1\ldots X_k\in B.$$
The subspace of $B$ generated by these elements is a subalgebra of $B$, denoted by $\WQSym(X)$. 
Abstracting this, we obtain an algebra $\WQSym$, with a basis $(M_w)$ indexed by the set of packed words. Its product is given by:
$$M_u M_v=\sum_{w \in \smop{QSh}(\max(u),\max(v))} M_{w\circ (u v[\max(u)])}.$$
\noindent The disjoint union of alphabets makes it a Hopf algebra, with the following coproduct:
$$\Delta(M_w)=\sum_{k=0}^{\max(w)} M_{w_{\mid \{1,\ldots,k\}}}\otimes M_{\smop{Pack}(w_{\mid\{k+1,\ldots,\max(w)\}})},$$
where for all set $I$, $w_{\mid I}$ is the word obtained by taking the letters of $w$ belonging to $I$.
The cartesian product of alphabets gives $\WQSym$ an internal coproduct:
$$\rho(M_u)=\sum_{i_1+\ldots+i_p=\max(u)}\sum_{v\in \smop{QSh}(i_1,\ldots,i_p)} M_{v\circ u}\otimes 
M_{(\underbrace{1\ldots 1}_{i_1}\ldots \underbrace{p\ldots p}_{i_p})\circ u}.$$

\subsection{The coalgebra species of set compositions}

We now define a bialgebra in the category of coalgebra species, which will give both $\QSym$ and $\WQSym$.

\begin{defn}\cite{Sta01}
Let $X$ be a finite set. A set composition or an ordered partition of $X$ is a finite sequence $(X_1,\ldots,X_k)$ of finite sets such that:
\begin{enumerate}
\item For all $1\leq i \leq k$, $X_i\neq \emptyset$.
\item $X=X_1\sqcup \ldots \sqcup X_k$.
\end{enumerate} 
For any finite space $X$, the space generated by the set of set compositions of $X$ will be denoted by $\setcomp_X$.
This defines a species $\setcomp$.
\end{defn}

The Hilbert formal series of $\setcomp$ is given by Fubini numbers, sequence A000670  of the OEIS.

We first give this species a structure of bialgebra in the category of species.

\begin{defn}\begin{enumerate}
\item Let $Y\subseteq X$ be two finite sets and let $C=(X_1,\ldots,X_k)$ be a set composition on $X$.
We put $I=\{i\in [k]\mid X_i\cap Y\neq \emptyset\}=\{m_1<\ldots<m_l\}$. The set composition $C_{\mid Y}$ of $Y$ is:
$$C_{\mid Y}=(X_{m_1}\cap Y,\ldots,C_{m_l}\cap Y).$$
For any finite sets $X,Y$, we define a product:
$$\left\{\begin{array}{rcl}
\setcomp_X \otimes \setcomp_Y&\longrightarrow&\setcomp_{X\sqcup Y}\\
C'\otimes C''&\longrightarrow&\displaystyle C'C''=\sum_{C,\,C_{\mid X}=C',C_{\mid Y}=C''} C.
\end{array}\right.$$
\item For any finite set $X$, we define a coproduct:
$$\Delta:\left\{\begin{array}{rcl}
\setcomp_X&\longrightarrow&\displaystyle \bigoplus_{Y\subseteq X} \setcomp_{X\setminus Y}\otimes \setcomp_Y\\
C=(X_1,\ldots,X_k)&\longrightarrow&\displaystyle \sum_{i=0}^k (X_1,\ldots,X_i)\otimes (X_{i+1},\ldots,X_k).
\end{array}\right.$$
\item For any finite set $X$, we define an internal coproduct $\rho$ on $\setcomp_X$, making it a coassociative, counitary coalgebra by:
\begin{align*}
\rho((X_1,\ldots,X_k))&=\sum_{i_1+\ldots+i_p=k} (X_1,\ldots,X_{i_1})\ldots (X_{i_1+\ldots+i_{p-1}+1},\ldots,X_{i_1+\ldots+i_p})\\
&\hspace{2.5cm}\otimes (X_1\sqcup\ldots\sqcup X_{i_1},\ldots,X_{i_1+\ldots+i_{p-1}+1}\sqcup \ldots \sqcup X_{i_1+\ldots+i_p}).
\end{align*}\end{enumerate}\end{defn}

{\bf Examples.} Let $A,B,C$ be finite, nonempty sets.
\begin{align*}
(A)(B)&=(A,B)+(B,A)+(A\sqcup B).\\
(A,B)(C)&=(A,B,C)+(A,C,B)+(C,A,B)+(A,B\sqcup C)+(A\sqcup C,B).\\
(A)(B,C)&=(A,B,C)+(B,A,C)+(B,C,A)+(A\sqcup B,C)+(B,A\sqcup C);\\ \\
\rho((A))&=(A)\otimes (A).\\
\rho((A,B))&=(A,B)\otimes (A\sqcup B)+(A)(B)\otimes (A,B)\\
&=(A,B)\otimes (A\sqcup B)+((A,B)+(B,A)+(A\sqcup B))\otimes (A,B).\\
\rho((A,B,C))&=(A,B,C)\otimes (A\sqcup B\sqcup C)+(A,B)(C)\otimes (A\sqcup B,C)\\
&+(A)(B,C)\otimes (A,B\sqcup C)+(A)(B)(C)\otimes (A,B,C)\\
&=(A,B,C)\otimes (A\sqcup B\sqcup C)\\
&+((A,B,C)+(A,C,B)+(C,A,B)+(A\sqcup C,B)+(A,B\sqcup C))\otimes (A\sqcup B,C)\\
&+((A,B,C)+(B,A,C)+(B,C,A)+(A\sqcup B,C)+(B,A\sqcup C))\otimes (A,B\sqcup C)\\
&+\big((A,B,C)+(A,C,B)+(B,A,C)+(B,C,A)+(C,A,B)+(C,B,A)\\
&+(A\sqcup B,C)+(A\sqcup C,B)+(B\sqcup C,A)+(A,B\sqcup C)+(B,A\sqcup C)+(C,A\sqcup B)\\
&+(A\sqcup B\sqcup C)\big)\otimes (A,B,C).
\end{align*}

%\diagramme{
%\xymatrix{
%\setcomp_X\otimes \setcomp_Y\otimes \setcomp Z\ar[d]^{I \otimes m} \ar[rrr]^{m\otimes I}&&&\setcomp_{X\sqcup Y}\otimes \setcomp_Z\ar[d]^m\\
%\setcomp_X\otimes \setcomp_{Y\sqcup Z}\ar[rrr]_{m}&&&\setcomp_{X\sqcup Y\sqcup Z}}}
%
%\diagramme{
%\xymatrix{
%\setcomp_{X}\ar [rrr]^\Delta\ar[d]_{\Delta}
%&&& \displaystyle\bigoplus_{Y\subset X}\setcomp_{X\setminus Y}\otimes\setcomp_{Y} \ar[d]^{I\otimes\Delta}\\
%\displaystyle\bigoplus_{Z\subset X}\setcomp_{X\setminus Z}\otimes\setcomp_{Z}
%\ar[rrr]_{\Delta\otimes I} &&&\displaystyle\bigoplus_{Z\subset Y\subset X}\setcomp_{X\setminus Y}\otimes\setcomp_{Y\setminus Z}\otimes\setcomp_Z
%}
%}
%\diagramme{
%\xymatrix{
%\setcomp_{X_1}\otimes\setcomp_{X_2}\ar [rr]^m\ar[d]_{\Delta\otimes\Delta}
%&& \setcomp_{X_1\sqcup X_2}\ar[d]^{\Delta}\\
%\displaystyle\bigoplus_{{\scriptstyle Y_1\subset X_1\atop \scriptstyle Y_2\subset X_2}}
%\setcomp_{X_1\setminus Y_1}\otimes\setcomp_{Y_1}\otimes\setcomp_{X_2\setminus Y_2}\otimes\setcomp_{Y_2}
%\ar[dr]_{\tau^{2,3}} &&\displaystyle\bigoplus_{Y\subset X_1\sqcup X_2}\setcomp_{(X_1\sqcup X_2)\setminus Y}\otimes\setcomp_{Y}\\
%& \displaystyle\bigoplus_{{\scriptstyle Y_1\subset X_1\atop \scriptstyle Y_2\subset X_2}}
%\setcomp_{X_1\setminus Y_1}\otimes\setcomp_{X_2\setminus Y_2}\otimes\setcomp_{Y_1}\otimes\setcomp_{Y_2}
%\ar[ur]_{m\otimes m}&
%}
%}

\begin{prop}
$\setcomp$ is a bialgebra in the category of coalgebra species.
\end{prop}

This proposition is a corollary of theorem \ref{theoextension} below, which will make $\setcomp$ appear as a quotient of the coalgebra species bialgebra $\topo$.\\

The counit of the coalgebra $\setcomp_X$ is given by:
$$\varepsilon(C)=\begin{cases}
1\mbox{ if }C=(X),\\
0\mbox{ otherwise}.
\end{cases}$$

We obtain from $\setcomp$ two bialgebras with an internal coproduct.
First, it induces a bialgebra structure on the vector space generated by the set compositions, up to a renumbering. 
For any set composition $C=(X_1,\ldots,X_k)$, we put $type(C)=(|X_1|,\ldots,|X_k|)$.
If $C,C'$ are two set compositions, $C$ and $C'$ are equal up to a renumbering if, and only if, $type(C)=type(C')$.
So this bialgebra has a basis $(M_c)$, indexed by compositions, and direct computations shows this is $\QSym$.

Secondly, we restrict ourselves to sets $[n]$, $n\geq 0$; we identify any subset $I\subseteq [n]$ with $[|I|]$ via the unique increasing bijection.
Set compositions on $[n]$ are identified with packed words of length $n$, via the bijection:
$$\left\{\begin{array}{rcl}
\{\mbox{Packed words of length }n\}&\longrightarrow&\setcomp_{[n]}\\
u&\longrightarrow&(u^{-1}(1),\ldots,u^{-1}(\max(u))).
\end{array}\right.$$
We obtain a bialgebra with a basis indexed by packed words, which is precisely $\WQSym$.

\subsection{Linear extensions}

\begin{defn}
Let $\T\in \topo_X$ and let $C=(X_1,\ldots X_k)\in \setcomp_X$.
We shall say that $C$ is a linear extension of $\T$ if :
\begin{enumerate}
\item For all $i,j\in [k]$, for all $x \in X_i$, $y\in X_j$, $x<_\T y$ $\Longrightarrow$ $i<j$.
\item For all $i,j\in [k]$, for all $x \in X_i$, $y\in X_j$, $x\sim_\T y$ $\Longrightarrow$ $i=j$.
\end{enumerate}
The set of linear extensions of $\T$ will be denoted by $\lin_\T$. \end{defn}

\begin{thm}\label{theoextension}
Let $X$ be a finite set. We define:
$$L:\left\{\begin{array}{rcl}
\topo_X&\longrightarrow&\setcomp_X\\
\T&\longrightarrow&\displaystyle \sum_{C \in \lin_\T}C.
\end{array}\right.$$
Then $L$ is a surjective morphism of bialgebras in the category of coalgebra species, that is to say:
\begin{enumerate}
\item For all finite sets $X,Y$, for all $\T\in \topo_X$, $\T'\in \topo_Y$, 
$$L(\T\T')=L(\T)L(\T').$$
\item For all finite set $X$, for all $\T\in \topo_X$,
$$\Delta\circ L(\T)=(L\otimes L)\circ \Delta(\T).$$
\item For all finite set $X$, for all $\T\in \topo_X$,
$$\rho\circ L(\T)=(L\otimes L)\circ \Gamma(\T).$$
 \end{enumerate}\end{thm}

\begin{proof} {\it First step}. Let us prove the following lemma: if $Y\subseteq X$, $\T\in \topo_X$ and $C \in \lin_\T$, then $C_{\mid Y}\in \lin_{\T_{\mid Y}}$.

We put $C=(X_1,\ldots,X_k)$ and $C_{\mid Y}=(X_{m_1}\cap Y,\ldots, X_{m_l}\cap Y)=(Y_1,\ldots,Y_l)$.
Let $i,j \in [l]$, $x \in Y_i$, $y\in Y_j$. If $x<_{\T\srestr Y} y$, then $x<_\T y$, so $m_i<m_j$, and finally $i<j$.
If $x\sim_{\T_{\mid Y}} y$, then $x\sim_\T y$, so $m_i=m_j$, and finally $i=j$.\\

 {\it Second step.} We prove (1). Let $\T\in \topo_X$ and $\T'\in \topo_Y$. Let us prove that:
$$\lin_{\T\T'}=\{C\in \setcomp_{X\sqcup Y} \mid C_{\mid X}\in \lin_\T,C_{\mid Y} \in \lin_{\T'}\}.$$
As $\T\T'_{\mid X}=\T$ and $\T\T'_{\mid Y}=\T'$, the first step implies that inclusion $\subseteq$ holds.
Moreover, if $x<_{\T\T'} y$ or $x\sim_{\T\T'} y$ in $X\sqcup Y$, then $(x,y) \in X^2$ or $(x,y)\in Y^2$, which implies the second inclusion.
Consequently:
\begin{align*}
L(\T\T')&=\sum_{C,\,C_{\mid X}\in \lin_\T,C_{\mid Y}\in \lin_{\T'}}C\\
&=\sum_{C'\in \lin_\T}\ \sum_{C''\in \lin_{\T'}} \ \sum_{C,\,C_{\mid X}=C', C_{\mid Y}=C''}C\\
&=\sum_{C'\in \lin_\T}\ \sum_{C''\in \lin_{\T'}} C'C''\\
&=L(\T)L(\T').
\end{align*} 

{\it Third step.} We prove (2). Let $\T$ be a topology on a set $X$. We put:
\begin{align*}
A&=\{(Y,C_1,C_2)\mid Y\in \T,C_1\in \lin_{\T_{\mid{X\setminus Y}}},C_2\in \lin_{\T_{\mid Y}}\},\\
B&=\{(C,i)\mid C\in \lin_\T,0\leq i\leq lg(C)\},
\end{align*}
which gives:
\begin{align*}
(L\otimes L)\circ \Delta(\T)&=\sum_{(Y,C_1,C_2)\in A}C_1\otimes C_2,\\
\Delta \circ L(\T)&=\sum_{((X_1,\ldots,X_k),i)\in B}(X_1,\ldots,X_i)\otimes (X_{i+1},\ldots,X_k).
\end{align*}
We define two maps:
\begin{align*}
f&:\left\{\begin{array}{rcl}
A&\longrightarrow&B\\
(Y,(X_1,\ldots,X_k),(X_{k+1},\ldots,X_{k+l}))&\longrightarrow&((X_1,\ldots,X_{k+l}),k),
\end{array}\right.\\ \\
g&:\left\{\begin{array}{rcl}
B&\longrightarrow&A\\
((X_1,\ldots,X_k),i)&\longrightarrow&(X_{i+1}\sqcup\ldots\sqcup X_i, (X_1,\ldots,X_i),(X_{i+1},\ldots,X_k)).
\end{array}\right. \end{align*}
Let us prove that $f$ is well-defined. If $(Y,C_1,C_2) \in A$, we put $C_1=(X_1,\ldots,X_k)$, $C_2=(X_{k+1},\ldots,X_{k+l})$, and $C=(X_1,\ldots,X_{k+l})$.
 Let us prove that $C\in \lin_{\T\T'}$. Let $x\in X_i$, $y\in X_j$. If $x<_\T y$, as $Y$ is an open set of $\T$, there are only three possibilities:
\begin{itemize}
\item $x,y \in Y$. As $C_2$ is a linear extension of $\T_{\mid Y}$, $i<j$.
\item $x,y \in X\setminus Y$. As $C_1$ is a linear extension of $\T_{\mid X\setminus Y}$, $i<j$.
\item $x\in X\setminus Y$ and $y\in Y$. Then $i\leq k<j$.
\end{itemize}
If  $x\sim_\T y$, as $Y$ is an open set of $\T$, so is a union of equivalence classes of $\sim_\T$, there are only two possibilities:
\begin{itemize}
\item $x,y \in Y$. As $C_2$ is a linear extension of $\T_{\mid Y}$, $i=j$.
\item $x,y \in X\setminus Y$. As $C_1$ is a linear extension of $\T_{\mid X\setminus Y}$, $i=j$.
\end{itemize}
So $f(Y,C_1,C_2)\in B$.

Let us prove that $g$ is well-defined. If $((X_1,\ldots,X_k),i)\in B$, we put $f((X_1,\ldots,X_k),i)=(Y,C_1,C_2)$.
$Y$ is an open set of $\T$: let $x \in Y$, $x\in X$, such that $x\leq_\T y$. We assume that $x\in X_j$, with $j\geq i$, and $y\in X_k$.
If $x \sim_\T y$, then $j=k\geq i$ and $y\in Y$. If $x<_\T y$, then $i\leq j<k$, so $y\in Y$. 
Moreover, $C_1=(X_1,\ldots,X_i)=C_{\mid X\setminus Y}$ and $C_2=(X_{i+1},\ldots,X_k)=C_{\mid Y}$. By the lemma of the first point,
$C_1\in \lin_{\T_{\mid X\setminus Y}}$ and $C_2\in \lin_{\T_{\mid Y}}$. So $(Y,C_1,C_2)\in A$. 

Moreover:
\begin{align*}
f\circ g((X_1,\ldots,X_k),i)&=f(X_{i+1}\sqcup \ldots \sqcup X_k,(X_1,\ldots,X_i),(X_{i+1},\ldots,X_k))\\
&=((X_1,\ldots,X_k),i);\\
g\circ f(Y,C_1,C_2)&=g(C_1.C_2,lg(C_1))\\
&=(Y,C_1,C_2).
\end{align*}
So $f$ and $g$ are bijections, inverse one from each other. Consequently:
\begin{align*}
(L\otimes L)\circ \Delta(\T)&=\sum_{(Y,C_1,C_2)\in A}C_1\otimes C_2\\
&=\sum_{((X_1,\ldots,X_k),i)\in B}(X_1,\ldots,X_i)\otimes (X_{i+1},\ldots,X_k)\\
&=\Delta\circ L(\T).
\end{align*}

{\it Fourth step.} Let  $A$ be the set of triples $(C,(i_1,\ldots,i_p),C')$ such that:
\begin{enumerate}
\item $C=(X_1,\ldots,X_k)$ and $C'=(X'_1,\ldots,X'_p)$ are set compositions of $X$, of respective length $k$ and $p$.
\item For all $j$, $i_j>0$ and $i_1+\ldots+i_p=k$.
\item For all $j$, $C'_{\mid X_{i_1+\ldots+i_{j-1}+1}\sqcup \ldots \sqcup X_{i_1+\ldots+i_j}}=(X_{i_1+\ldots+i_{j-1}+1},\ldots,X_{i_1+\ldots+i_j})$.
\end{enumerate}
Let $B$ be the set of triples $(\T',C',C'')$ such that:
\begin{enumerate}
\item $\T'\oprec \T$.
\item $C'$ is a linear extension of $\T'$.
\item $C''$ is a linear extension of $\T/\T'$.
\end{enumerate}
Then:
\begin{align*}
\rho \circ L(\T)&=\sum_{((X_1,\ldots,X_p),(i_1,\ldots,i_p),C') \in A} C'\otimes
(X_1\sqcup \ldots \sqcup X_{i_1},\ldots,X_{i_1+\ldots+i_{p-1}+1}\sqcup \ldots \sqcup  X_{i_1+\ldots+i_p}),\\
(L\otimes L)\circ \Gamma(\T)&=\sum_{(\T,C,C')\in B}C'\otimes C''.
\end{align*}
We now prove the following lemma: if $(\T,C',C'') \in B$, with $C''=(X''_1,\ldots,X''_q)$, then:
$$\T'=\T_{\mid X''_1}\ldots \T_{\mid X''_q}.$$
We first show that for all $i$, $\T'_{\mid X''_i}=\T_{\mid X''_i}$. Let us assume that $x,y \in X''_i$, such that $x\leq_\T y$.
Then $x\leq_{\T/\T'} y$. If $x<_{\T/\T'} y$, as $C''$ is a linear extension of $\T/\T'$, we would have $x \in X''_a$, $y\in X''_b$, with $a<b$:
this is a contradiction. So $x\sim_{\T/\T'} y$. As $\T'\oprec \T$, $x\sim_{\T'/\T'} y$, so $x$ and $y$ are in the same connected component $Y$ of $\T'$.
As $\T'\oprec \T$, $x\leq_{\T_{\mid Y}} y$, so $x\leq_{\T'_{\mid Y}}y$, so $x\leq_{\T'} y$. Conversely, if $x\leq_{\T'} y$, as $\T'\prec \T$,
$x\leq_\T y$. 

Let $x \in X''_i$, $y\in X''_j$, with $i<j$. As $C''$ is a linear extension of $\T/\T'$, we do not have $x\sim_{\T/\T'} y$, and, as $\T'\oprec \T$,
we do not have $x\sim_{\T'/\T'} y$. Consequently:
$$\T'=\T'_{\mid X''_1}\ldots \T'_{\mid X''_q}=\T_{\mid X''_1}\ldots \T_{\mid X''_q}.$$

{\it Fifth step.} We prove (3). We define a map $f:A\longrightarrow B$ by $f(C,(i_1,\ldots,i_p),C')=(\T',C',C'')$, where:
\begin{enumerate}
\item $C''=(X_1\sqcup \ldots \sqcup X_{i_1},\ldots,X_{i_1+\ldots+i_{p-1}+1}\sqcup \ldots \sqcup  X_{i_1+\ldots+i_p})$.
\item $\T'=\T_{\mid X''_1}\ldots \T_{\mid X''_p}$.
\end{enumerate}
Let us prove that $f$ is well-defined. First, $\T'\prec \T$. If $Y\subseteq X$ is connected for $\T'$, then necessarily there exists a $i$,
such that $Y\subseteq X''_i$. Then $\T'_{\mid Y}=(\T'_{\mid X''_i})_{\mid Y}=(\T_{\mid X''_i})_{\mid Y}=\T_{\mid Y}$.

Let us assume that $x\sim_{\T/\T'} y$. There exists a sequence of elements of $x$ such that:
$$x\leq_\T x_1 \geq_{\T'} y_1 \leq_\T x_2 \geq_{\T'}\ldots \leq_\T x_r \geq_{\T'} y.$$
 If $y_a \in X''_j$, as $C$ is a linear extension of $\T$, necessarily $x_{a+1} \in X''_k$, with $k\geq j$.
 If $x_a \in X''_j$, as $x_a \geq_{\T'} y_a$, $y_a \in X''_j$. Consequently, if $x \in X''_i$, then 
$x_1,y_1,\ldots,x_r,y \in X''_i\sqcup \ldots X''_p$. By symmetry of $x$ and $y$, $x,x_1,y_1,\ldots,x_r,y \in X''_i$. So, by restriction to $X''_i$:
$$x\leq_{\T'} x_1 \geq_{\T'} y_1 \leq_{\T'} x_2 \geq_{\T'}\ldots \leq_{\T'} x_r \geq_{\T'} y.$$
This gives $x \sim_{\\T'/\T'} y$: we finally obtain that $\T'\oprec \T$.\\

By the lemma of the first step, $C_{\mid X''_i}$ is a linear extension of $\T_{\mid X''_i}$, so, by definition of $A$,
$C'$ is a linear extension of $\T_{\mid X''_1}\ldots \T_{\mid X''_p}=\T'$. \\

Let us assume that $x<_{\T/\T'} y$. Let $i,j$ such that $x \in X''_i$, $y\in X''_j$. 
Up to a change of $x\in X''_i$, $y\in X''_j$, we can assume that $x<_\T y$.
If $i=j$, then by restriction $x<_{\T'} y$, so $x\sim_{\T'/\T'} y$ and finally $x\sim_{\T/\T'} y$, as $\T'\oprec \T$: this is a contradiction.
Hence, $i\neq j$, and $x<_\T y$; as $C$ is linear extension of $\T$, necessarily $i<j$.

Let us assume that $x\leq_{\T/\T'} y$. Let $i,j$ such that $x \in X''_i$, $y\in X''_j$. By definition of $\leq_{\T/\T'}$, we can assume that
$x\leq_\T y$ or $x\sim_{\T'} y$. In the first case, as $C$ is a linear extension of $\T$, we have $x \in X_a$, $y\in X_b$, with $a\leq b$,
so $i\leq j$. In the second case, $i=j$. Consequently, if $x\sim_{\T/\T'} y$, then $i\leq j$ and $j\leq i$, so $i=j$.
We proved that $C'' \in \lin_{\T/\T'}$.\\

We now consider the map $g:B\longrightarrow A$, defined  by $g(\T',C',C'')=(C,(i_1,\ldots,i_p),C')$, with:
\begin{enumerate}
\item $C=C'_{\mid X''_1}\ldots C'_{\mid X''_p}$, if $C''=(X''_1,\ldots,X''_p)$.
\item For all $j$, $i_j=|X''_j|$.
\end{enumerate}
Let us prove that $g$ is well-defined. Let us assume $x<_\T y$, with $x \in X''_i$, $y\in X''_j$. Let $a,b$ such that $x \in X_a$, $y\in X_b$,
if $C=(X_1,\ldots,X_k)$. If $i=j$, then by the lemma of the fourth step, $x<_{\T'} y$. As $C'$ is a linear extension of $\T'$,
$x \in C'_c$, $y\in C'_d$, with $c<d$. By definition of $C$, $a<b$. If $i\neq j$, then $x\leq_{\T/\T'} y$; as $C''$ is a linear extension
of $\T/\T'$, $i<j$, so $a<b$. 
If $x\sim_\T y$, a similar argument proves that $x,y\in X_a$ for a certain $a$. So $C \in \lin_\T$. Moreover, for all $j$:
$$C'_{\mid C_{i_1+\ldots+i_{j-1}+1}\sqcup\ldots \sqcup C_{i_1+\ldots+i_j}}
=C'_{\mid C''_j}=C_{\mid C''_j}=(C_{i_1+\ldots+i_{j-1}+1},\ldots,C_{i_1+\ldots+i_j}).$$
So $g$ is well-defined. The lemma of the fourth step implies that $f\circ g=Id_B$, and by definition of $A$, $g\circ f=Id_A$, so
$f$ and $g$ are bijective, inverse one from each other. Finally:
\begin{align*}
(L\otimes L)\circ \Gamma(\T)&=\sum_{(\T',C',C'')\in B} C'\otimes C''\\
&=\sum_{((X_1,\ldots,X_p),(i_1,\ldots,i_p),C') \in A} C'\otimes
(X_1\sqcup \ldots \sqcup X_{i_1},\ldots,X_{i_1+\ldots+i_{p-1}+1}\sqcup \ldots \sqcup  X_{i_1+\ldots+i_p})\\
&=\rho \circ L(\T).
\end{align*}
 
 {\it Last step.} It remains to prove the surjectivity of $L$. Let $(X_1,\ldots,X_k)$ be a set composition of $X$. Let $\T$ be the topology whose open sets
 are $X_i\sqcup\ldots X_k$, for $1\leq i\leq k$, and $\emptyset$. Then $\T$ has a unique linear extension, which is $C$, so $L(\T)=C$. \end{proof}

{\bf Examples.} If $X=E\sqcup F=A\sqcup A\sqcup C$ are two partitions of $X$:
\begin{align*}
L(\tdun{$X$})&=(X),\\
L(\tddeux{$E$}{$F$})&=(E,F),\\
L(\tdun{$E$}\tdun{$F$})&=(E,F)+(F,E)+(E\sqcup F),\\
L(\tdtroisun{$A$}{$C$}{$B$})&=(A,B,C)+(A,C,B)+(A,B\sqcup C),\\
L(\tdtroisdeux{$A$}{$B$}{$C$})&=(A,B,C),\\
L(\pdtroisun{$A$}{$B$}{$C$})&=(B,C,A)+(C,B,A)+(B\sqcup C,A),\\
L(\tddeux{$A$}{$B$}\tdun{$C$})&=(A,B,C)+(A,C,B)+(C,A,B)+(A\sqcup C,B)+(A,B\sqcup C),\\
L(\tdun{$A$}\tdun{$B$}\tdun{$C$})&=(A,B,C)+(A,C,B)+(B,A,C)+(B,C,A)+(C,A,B)+(C,B,A)\\
&+(A\sqcup B,C)+(A\sqcup C,B)+(B\sqcup C,A)+(A,B\sqcup C)+(B,A\sqcup C)+(C,A\sqcup B)\\
&+(A\sqcup B\sqcup C).
\end{align*}

Now we consider isomorphism classes of finite topologies and set compositions.
Let $\T$ be a topology on a finite set $X$, and let $Z$ be an infinite, totally ordered alphabet. A linear extension of $\T$ is map $f:X\longrightarrow Z$,
such that:
\begin{enumerate}
\item $x<_\T y$ in $X$ $\Longrightarrow$ $f(x)<f(y)$.
\item $w\sim_\T y$ in $X$ $\Longrightarrow$ $f(x)=f(y)$.
\end{enumerate}
The set of linear extensions of $\T$ with values in $Z$ is denoted by $\lin_\T(Z)$.
\begin{thm}
Let $Z$ be an infinite, denumerable, totally ordered alphabet. Identifying $\QSym(Z)$ and $\QSym$, we define a map:
$$\lambda:\left\{\begin{array}{rcl}
\h&\longrightarrow&\QSym\\
\T\in \topo_X&\longrightarrow&\displaystyle \sum_{f\in \lin_\T(Z)} \prod_{x\in X} f(x).
\end{array}\right.$$
Then $\lambda$ is a Hopf algebra morphism, compatible with internal coproducts of $\h$ and $\QSym$.
\end{thm}

{\bf Examples.}
\begin{align*}
\lambda(\tdun{$a$})&=M_{(a)},\\
\lambda(\tddeux{$a$}{$b$})&=M_{(a,b)},\\
\lambda(\tdun{$a$}\tdun{$b$})&=M_{(a,b)}+M_{(b,a)}+M_{(a+b)},\\
\lambda(\tdtroisun{$a$}{$c$}{$b$})&=M_{(a,b,c)}+M_{(a,c,b)}+M_{(a,b+c)},\\
\lambda(\tdtroisdeux{$a$}{$b$}{$c$})&=M_{(a,b,c)},\\
\lambda(\pdtroisun{$a$}{$b$}{$c$})&=M_{(b,c,a)}+M_{(c,b,a)}+M_{(b+c,a)},\\
\lambda(\tddeux{$a$}{$b$}\tdun{$c$})&=M_{(a,b,c)}+M_{(a,c,b)}+M_{(c,a,b)}+M_{(a+c,b)}+M_{(a,b+c)},\\
\lambda(\tdun{$a$}\tdun{$b$}\tdun{$c$})&=M_{(a,b,c)}+M_{(a,c,b)}+M_{(b,a,c)}+M_{(b,c,a)}+M_{(c,a,b)}+M_{(c,b,a)}\\
&+M_{(a+b,c)}+M_{(a+c,b)}+M_{(b+c,a)}+M_{(a,b+c)}+M_{(b,a+c)}+M_{(c,a+b)}+M_{(a+b+c)}.
\end{align*}

\noindent Restricting to finite topologies and set compositions on sets $[n]$, we obtain the following theorem:
\begin{thm}
Let $Z$ be an infinite, totally ordered alphabet. Identifying $\WQSym(Z)$ and $\WQSym$, we define a map:
$$\Lambda:\left\{\begin{array}{rcl}
\h_{\mathbf{T}}&\longrightarrow&\WQSym\\
\T\in \topo_{[n]}&\longrightarrow&\displaystyle \sum_{f\in \lin_\T(Z)} f(1)\ldots f(n).
\end{array}\right.$$
Then $\Lambda$ is a Hopf algebra morphism, compatible with internal coproducts of $\h_{\mathbf{T}}$ and $\WQSym$.
\end{thm}

{\bf Examples.}
\begin{align*}
\Lambda(\tdun{$1$})&=M_{(1)},\\
\Lambda(\tddeux{$1$}{$2$})&=M_{(1,2)},\\
\Lambda(\tddeux{$2$}{$1$})&=M_{(2,1)},\\
\Lambda(\tdun{$1$}\tdun{$2$})&=M_{(1,2)}+M_{(2,1)}+M_{(1,1)},\\
\Lambda(\tdtroisun{$1$}{$2$}{$3$})&=M_{(1,2,3)}+M_{(1,3,2)}+M_{(1,2,2)},\\
\Lambda(\tdtroisun{$2$}{$3$}{$1$})&=M_{(2,1,3)}+M_{(3,1,2)}+M_{(2,1,2)},\\
\Lambda(\tdtroisun{$3$}{$2$}{$1$})&=M_{(2,3,1)}+M_{(3,2,1)}+M_{(2,2,1)},\\
\Lambda(\tdtroisdeux{$1$}{$2$}{$3$})&=M_{(1,2,3)},\\
\Lambda(\tdtroisdeux{$2$}{$3$}{$1$})&=M_{(3,1,2)},\\
\Lambda(\tdtroisdeux{$3$}{$1$}{$2$})&=M_{(2,3,1)},\\
\Lambda(\pdtroisun{$1$}{$2$}{$3$})&=M_{(3,1,2)}+M_{(3,2,1)}+M_{(2,1,1)},\\
\Lambda(\pdtroisun{$2$}{$1$}{$3$})&=M_{(1,3,2)}+M_{(2,3,1)}+M_{(1,2,1)},\\
\Lambda(\pdtroisun{$3$}{$1$}{$2$})&=M_{(1,2,3)}+M_{(2,1,3)}+M_{(1,1,2)},\\
\Lambda(\tddeux{$1$}{$2$}\tdun{$3$})&=M_{(1,2,3)}+M_{(1,3,2)}+M_{(2,3,1)}+M_{(1,2,1)}+M_{(1,2,2)},\\
\Lambda(\tdun{$1$}\tdun{$2$}\tdun{$3$})&=M_{(1,2,3)}+M_{(1,3,2)}+M_{(2,1,3)}+M_{(2,3,1)}+M_{(3,1,2)}+M_{(3,2,1)}\\
&+M_{(1,1,2)}+M_{(1,2,1)}+M_{(2,1,1)}+M_{(1,2,2)}+M_{(2,1,2)}+M_{(2,2,1)}+M_{(1,1,1)}.
\end{align*}

\section{Quasi-ormould composition}\label{sect:moulds}
%%%%%

A mould is a collection $M_{\bullet} = \{
M^{\tmmathbf{\omega}} \}$ of elements of some commutative algebra
$\mathcal{A}$, indexed by finite sequences $\tmmathbf{\omega}= ( \omega_1,
\ldots, \omega_r)$ of elements of a set $\Omega$; equivalently, it is an
$\mathcal{A}$--valued function on the set of words $\omega_1 \ldots \omega_r$
in the alphabet $\Omega$. In what follows, the alphabet is in fact the
underlying set of an additive semi--group, a typical example in the
applications being the set of positive integers $\Omega =\mathbb{N}_{>
0} \nosymbol$. We can already notice that from the outset, moulds involve
combinatorial objects which are both labelled and decorated: the labels are
integers belonging to some $[ r]$ and the decorations belong to $\Omega$. When
the values of a mould $M^{\bullet}$ are in fact independant from any set
$\Omega$, $M^{\bullet}$ is said to be {\tmem{of constant type.}}

In the context in which they originated, namely the classification of
dynamical sytems, they naturally appear matched with dual objects, named
comoulds, in expansions of the following form:

\[ F = \sum M^{\tmmathbf{\omega}} B_{\tmmathbf{\omega}} = \sum_{r \geqslant 0}
\sum_{\mathbf{\omega}
= ( \omega_1, \ldots, \omega_r)} %\tmtextbf{}
 M^{\tmmathbf{\omega}} B_{\tmmathbf{\omega}} \]

A comould $B_{\bullet} = \{ B_{\tmmathbf{\omega}} \}$ is a collection, indexed
by sequences $\tmmathbf{\omega}$ as above, of elements of some bialgebra $(
\mathcal{B}, +, ., \sigma)$ , and such expansions, known as mould--comould
{\tmem{contractions}}, make sense in the completed algebra spanned by the
$B_{\tmmathbf{\omega}}$, with respect to the gradings given by the length of
sequences (other gradings may be relevant). In most situations, the
$B_{\tmmathbf{\omega}}$ are products of some building blocks $B_{\omega}$
($\omega \in \Omega$) : $B_{\tmmathbf{\omega}} = B_{\omega_r} \ldots
B_{\omega_1}$ ; the building blocks themselves are abstracted from the
dynamical system under study and are mapped to ordinary differential operators
acting on spaces of formal series, through some evaluation morphism (\cite{Eca92,
FM14}).

Accordingly, these expansions can be realized as elements of completions of
huge linear spaces of operators, typically $\tmop{End} ( \mathbb{C} [ [
x]])$, and they are naturally endowed with a linear structure and two
non--linear operations, a product $\times$ and a composition product $\circ$.
Indeed, for a given comould $B_{\bullet}$ and two moulds $M^{\bullet}$ and
$N^{\bullet}$, the product of the operators associated respectively to
$N^{\bullet}$ and $M^{\bullet}$ can be expanded as a contraction with
$B_{\bullet}$, yielding a new mould $P^{\bullet} = M^{\bullet} \times
N^{\bullet}$ :
\[ \left( \sum N^{\tmmathbf{\omega}} B_{\tmmathbf{\omega}} \right) \left( \sum
   M^{\tmmathbf{\omega}} B_{\tmmathbf{\omega}} \right) = \sum
   P^{\tmmathbf{\omega}} B_{\tmmathbf{\omega}} = \sum_{r \geqslant 0}
   \sum_{\tmmathbf{\omega}= ( \omega_1, \ldots, \omega_r)} %\tmtextbf{}
   P^{\tmmathbf{\omega}} B_{\tmmathbf{\omega}} \]
and the formula giving the components of the product mould $P_{\bullet}$ is as
follows:
\[ P^{( \omega_1, \ldots, \omega_r)} = \sum M^{( \omega_1, \ldots, \omega_i)}
   N^{( \omega_{i + 1}, \ldots, \omega_r)} \]
The product is obviously associative, non commutative in general, and
distributive over the sum.

Beside this product of operators, we can also use some given mould $M$
{\tmem{to change the set of letters}} $B_{\tmmathbf{\omega}}$ and this will
give us the composition $\circ$ of moulds, $B_{\bullet} \longrightarrow
C_{\bullet} \nocomma \tmop{with}$:
\[ C_{\omega_0} = \sum_{\| \tmmathbf{\omega} \| = \omega_0}
   M^{\tmmathbf{\omega}} B_{\tmmathbf{\omega}} \]
where the norm of the sequence $\tmmathbf{\omega}$ is by definition $\| (
\omega_1, \ldots, \omega_r) \| = \omega_1 + \ldots + \omega_r$ .

Performing this natural change of alphabets \ successively with two moulds
$M^{\bullet}$ and $N^{\bullet}$ , amounts to a change of alphabet with respect
to a mould $Q^{\bullet} = M^{\bullet} \circ N^{\bullet}$ which is given by:
\[ Q^{( \omega_1, \ldots, \omega_r)} = \sum M^{( \| \tmmathbf{\omega}^1 \|,
   \ldots \| \tmmathbf{\omega}^s \|)} N^{\tmmathbf{\omega}^1} \ldots
   N^{\tmmathbf{\omega}^s} \]
the sum being performed over all the ways of obtaining the sequence
$\tmmathbf{\omega}$ by concatenation of the subsequences $\tmmathbf{\omega}^i$
: $\tmmathbf{\omega}=\tmmathbf{\omega}^1 \ldots \tmmathbf{\omega}^s$ .

The composition product is also associative, non commutative in general, and
right--distributive over the sum and product. It is worth noticing that the
operation of {\tmem{mould composition}} involve {\tmem{compositions
of integers}}, which is unsurprising if we think that the
operations on constant--type moulds behave exactly as the sums, product and
compositions of formal series.

In practice, the building blocks $B_{\omega}$ are \ such that their coproducts
generally fall into two categories, according to the nature of the system: the
$B_{\omega}$ are derivations ($\sigma ( B_{\omega}) = B_{\omega} \otimes 1 + 1
\otimes B_{\omega} $) in the case of dynamical systems with continuous time
and of divided powers type ($\sigma ( B_{\omega}) = \sum_{\omega_1 +
\omega_2 = \omega} B_{\omega_1} \otimes B_{\omega_2} $) in the case of
dynamical systems with dicrete time. In the first case, the comould is called
cosymmetral and in the second one cosymmetrel.

The expectation for mould--comould contractions $F$ to have good algebraic
properties, namely being automorphisms or derivations ($\sigma ( F) = F
\otimes F$ or $\sigma ( F) = F \otimes 1 \noplus + 1 \otimes F$), imposes
symmetry constraints on the moulds which are matched with a given cosymmetral
or cosymmetrel comould, and this directly leads to the following:

\begin{defn}

  A mould $M^{\bullet}$ is called symmetral (resp. alternal) iff
  $M^{\emptyset} = 1$(resp. $0$) and
  \[ M^{\tmmathbf{\omega}^1} M^{\tmmathbf{\omega}^2} =
     \sum_{\tmmathbf{\omega}\in \tmop{Sh} ( \tmmathbf{\omega}^1,
     \tmmathbf{\omega}^2)} M^{\tmmathbf{\omega}} \]
       (resp.
    $$   \sum_{\tmmathbf{\omega}\in \tmop{Sh} ( \tmmathbf{\omega}^1,
     \tmmathbf{\omega}^2)} M^{\tmmathbf{\omega}}=0$$
      if $\tmmathbf{\omega}^1,\tmmathbf{\omega}^2\neq \emptyset$).
  
  A mould $M^{\bullet}$ is called symmetrel (resp. alternel) iff
  $M^{\emptyset} = 1$(resp. $0$) and
  \[ M^{\tmmathbf{\omega}^1} M^{\tmmathbf{\omega}^2} =
     \sum_{\tmmathbf{\omega}\in \tmop{QSh} ( \tmmathbf{\omega}^1,
     \tmmathbf{\omega}^2)} M^{\tmmathbf{\omega}} \]
     (resp.
    $$   \sum_{\tmmathbf{\omega}\in \tmop{QSh} ( \tmmathbf{\omega}^1,
     \tmmathbf{\omega}^2)} M^{\tmmathbf{\omega}}=0$$
      if $\tmmathbf{\omega}^1,\tmmathbf{\omega}^2\neq \emptyset$).

  \ \ $ \tmop{Sh} ( \tmmathbf{\omega}^1,
  \tmmathbf{\omega}^2)$ (resp. $\tmop{QSh} (  \tmmathbf{\omega}^1, \tmmathbf{\omega}^2)$ ) designates all the sequences
  $\tmmathbf{\omega}$ that can be obtained by shuffling (resp. quasishuffling)
  the two sequences $\tmmathbf{\omega}^1$ and $\tmmathbf{\omega}^2$.
\end{defn}

When viewed as linear applications, symmetral (resp. symmetrel) moulds are
characters of the decorated shuffle algebra (resp quasishuffle algebra) and
alternal (resp. alternel moulds) are infinitesimal characters of these Hopf
algebras.

Most properties of stability immediately follow from the definitions and are
summed up in the:

\begin{prop}
  {\tmdummy}
  
  \begin{enumerate}
    \item Symmetral $\times $ Symmetral = Symmetral
  
   \item Symmetrel $\times $ Symmetral = Symmetral
   
    \item Alternal $\circ$ Alternal = Alternal
    
    \item Symmetral $\circ$ Alternal = Symmetral
    
    \item Symmetrel $\circ$ Symmetral = Symmetral
    
    \item Alternel $\circ$ Symmetral = Alternal
    
    \item Alternal $\circ$ Alternel = Alternel
    
    \item Symmetral $\circ$ Alternel = Symmetrel
    
    \item Symmetrel $\circ$ Symmetrel = Symmetrel
    
    \item Alternel $\circ$ Symmetrel = Alternel
  \end{enumerate}
\end{prop}

In particular, symmetrel moulds are stable by both the product $\times$ and
the composition product, a fact which amounts to a statement on convolution of
characters for the corresponding coproducts on the algebra based on
$\Omega$--decorated words.

Next, to tackle difficult questions of analytic classification, J. Ecalle had
been driven to reorder mould--comould contractions by a systematic use of trees
(\cite{Eca92}), by considering so--called arborescent moulds,  armoulds for short,
which are indexed by sequences with arborescent partial orders (each element
has at most one antecedent) on the labelling sets $[ r]$.

In this context, the product of armoulds is nothing but the convolution with
respect to Connes--Kreimer coproduct, when separative armoulds are seen as
characters of the relevant Hopf algebra on trees (\cite{FM14}). There is also a
natural definition of composition of armoulds (it appears in particular in
\cite{Men06}), related to another coproduct on the algebra of decorated forests, which
is a decorated version of the coproduct introduced and studied in \cite{CEM11} (see also \cite{Man12}) and
which corresponds to the operation of substitution in the domain of B--series.
This last coproduct involves suppression of edges on a given tree and a notion
of quotient tree which is the one that was to be conveniently generalized to
partial orders and finally to quasi--orders in the present text.

In fact, as mentioned e.g. in the paper \cite{EV04}, the natural operations $+,
\times, \circ$ on (ordinary) moulds and armoulds can be extended to moulds
associated to sequences with a general partial order, the name ormoulds being
coined for such objects by J. Ecalle. An ormould $M^{\sharp}$, with values in the
commutative algebra $\mathcal{A}$, and indexed by elements of a semi--group
$\Omega$, is a collection of elements of $\mathcal{A}$ indexed by orsequences
$\tmmathbf{\omega}^{\sharp}$, namely sequences $\tmmathbf{\omega}= ( \omega_1,
\ldots, \omega_r)$ of elements of $\Omega$ endowed with a quasi--order on the
labelling set $[ r]$. It is indeed possible to give (\cite{Eca_Ormoulds}) quite
natural definitions for the product and composition product of ormoulds,
involving the concept of ``orderable partition'' of a poset
and the general study of ormoulds with Hopf algebraic techniques will be the object
of a separate article.

Moreover, as it is already the case for arborescent moulds, there is a natural
``disordering morphism'' from ordinary moulds to ormoulds, which exists in two
versions, a simple and contracting one, adapted to symmetral and symmetrel
moulds respectively. Unsurprisingly, it is naturally defined in terms of
linear extensions and is coherent with the constructions of
Malvenuto--Reutenauer done in \cite{MR11} in the undecorated case.

Now, all these definitions, constructions, symmetries and operations on
moulds, that have been introduced and exposed by J. Ecalle are in fact very
``robust'' : they can even actually be pushed one step forward, from posets to
quasi--posets, starting with the very natural:

\begin{defn}
  A quasi--ormould $M^{\dashv}$ , with values in the commutative algebra
  $\mathcal{A}$, and indexed by elements of a semi--group $\Omega$, is a
  collection of elements of $\mathcal{A}$ indexed by sequences
  $\tmmathbf{\omega}^{\dashv}$ of elements of $\Omega$ endowed with a
  quasi--order on the labelling set $[ r]$.
\end{defn}

In conformity with J. Ecalle's terminology, we shall call quasi--orsequence the
data of a sequence \ $( \omega_1, \ldots, \omega_r)$, with a quasi--order on
the set $[ r]$; equivalently, a quasi--orsequence is a labelled and
$\Omega$--decorated quasi--poset. A quasi--ormould \ can thus be seen as a
($\mathcal{A}$--valued) linear map on the vector space spanned by
quasi--orsequences and, with this definition, a {\tmem{separative
quasi--ormould}} is a $\mathcal{A}$--valued character of the
($\Omega$--decorated version of the) Hopf algebra $\mathcal{H}_T$. To
conclude, we now on focus on constant--type quasi--ormoulds, in other words
characters of of the Hopf algebra $\mathcal{H}_T$. The natural definition of
the product of two separative quasi--ormoulds is then nothing but the
convolution for the external coproduct $\Delta$, when seen as characters of
the Hopf algebra $\mathcal{H}_T$. The definition of the composition of
quasi--ormoulds itself amounts to the translation of the coproduct $\Gamma$
given above.

The commutativity of the diagram of Theorem \ref{structures} is equivalent to the fact that the composition of separative quasi-ormoulds is distributive with respect to the product:

\begin{equation}
(M_1M_2)\circ N=(M_1\circ N)(M_2\circ N).
\end{equation}

At this stage, and building on the previous constructions, we have at our
disposal a natural application from moulds to quasiormoulds that is the
natural generalization of the previous constructions of J. Ecalle : the
quasi--posetization $\mathfrak{Q}$ of a (constant type) symmetrel mould,
viewed as a character $\varphi$ on $\QSym$ is the character on $\mathcal{H}$
given by $\mathfrak{Q} ( \varphi) = \varphi \circ \lambda$ and theorem can
be rephrased in the mould formalism by stating that quasi--posetization
respects the product and composition product of quasi--ormoulds.

%%%%%%%%%%%%%%%%%%%%%%%%%%%%%%%%%%%%%%%%%%%%%%%%%%%%%%%%%%%%%%%%%%%%

\section{Outlook}\label{sect:concl}

 The consideration of some of the basic constructions of J. Ecalle's mould
calculus, enhanced at the level of quasi--posets has thus led to the
construction of a new internal coproduct which interacts in a nice way with the
algebraic structures recently introduced in \cite{FMP} and \cite{FM}.
We then built, through the formalism of linear species, natural morphisms that respect all
the structures involved, from the algebras of quasi--posets to the algebras
$\QSym$ and $\WQSym$, which are pervasive objects by now in combinatorics and other
fields.

Which use can be made of this internal coproduct, and specifically in
interaction with the external one, to investigate properties of finite
topological spaces or to tackle purely combinatorial questions is at this
stage still open, but in its original field of application, the rich algebraic
structures of mould calculus have already made it possible to treat questions
that appear to be out of scope of other techniques (\cite{Eca92, Men06, Men07}).

Actually, two distinct features give striking efficiency of mould calculus in the domain of dynamical
sytems (and more recently for the study of Multiple Zeta Values):

-- The existence of several interacting operations (many more are contained in J. Ecalle's papers),

-- The existence of a collection of particular moulds, of constant use, with
closed--form expression.

In fact, some isolated examples of characters have appeared lately in algebraic
combinatorics or in the algebraic study of control theory (see e. g. the
mentions of works by A. Murua or F. Chapoton and others, recalled in \cite{CEM11}),
which are particular to moulds introduced and tabulated by J. Ecalle some time
ago. A striking example of the existence of some closed--form character,
completely independently from J. Ecalle's formalism but most similar, was
produced by J. Unterberger and the second author in the field of rough paths, see \cite{FU} and the
references therein. 
%A more systematic study of the available
%operations, and particular mould composition, by combining suitably chosen
%characters of Hopf algebras might eventually yield interesting results within
%algebraic combinatorics, but also in numerical analysis, control theory and
%beyond.

Another natural question is to study the internal products that might be
counterparts of the internal coproducts constructed in the present text and to
study them in connection with some products existing in the litterature (such
as the ones on the algebra $\mathbf{PQSym}$ of Parking Quasi Symmetric Functions of
Thibon et al); these questions are not
straightforward, as the duality on quasi--posets is degenerate
\cite{Foi13_Adv, FM}.

Finally, as recalled in the text for $\QSym$ and $\WQSym$, many interesting
combinatorial Hopf algebras have polynomial realizations, in which the basis
elements are realized as polynomials in an auxiliary set of commuting or non
commuting variables. Such presentations have many advantages, beyond e.g. the
very fast way of proving coassociativity by the doubling of alphabet trick
implemented above. Polynomial realizations were recently obtained in \cite{FNT14}
for the algebra of labelled forests and several related Hopf algebras: the
extensions of the ideas of \cite{FNT14} to the posets and quasi--posets Hopf algebras
remains to be done. The existence of a polynomial
realization would be expected, as the internal coproduct of $\WQSym$ is exactly the one induced by the
cartesian product of alphabets.

%%%%%%%%%%%%%%%%%%%%%%%%%%%%%%%%%%%%%%%%%%%%%%%%%%%%%%%%%%%%%%%%%%%%

\end{document}